\documentclass[reqno,a4paper,11pt]{amsart}
\usepackage{blindtext}

\newcommand{\rel}{\mathfrak{R}}

\usepackage{hyperref}

\usepackage{enumerate}
\usepackage{verbatim}
\usepackage{pdfsync}

\usepackage[all]{xy}
\usepackage[usenames,dvipsnames]{xcolor}
\usepackage{tikz}
\usepackage{tikz-cd}

\usepackage{pgfplots}
\usetikzlibrary{3d}
\usetikzlibrary{shapes,decorations.pathreplacing}
\usetikzlibrary{calc}
\usetikzlibrary{backgrounds}
\usetikzlibrary{automata}
\usetikzlibrary{positioning}
\usetikzlibrary{decorations.markings}
\usetikzlibrary{arrows}
\usetikzlibrary{scopes}
\usetikzlibrary{intersections}
\tikzstyle myBG=[line width=3pt,opacity=1]

\newcommand{\lb}{\langle}
\newcommand{\rb}{\rangle}

\newcommand{\ZE}{\mathbb{Z}E}
\newcommand{\ZV}{\mathbb{Z}V}

\newcommand{\gr}{\mathrel{\mathscr{R}}}

\newcommand{\ZL}{{\mathbb{Z}L}}

\newcommand{\ZM}{{\mathbb{Z}M}}

\newcommand{\Z}{\mathbb{Z}}

\newcommand{\FPn}{{\rm FP}\sb n}
\newcommand{\Fn}{{\rm F}\sb n}
\newcommand{\F}{{\rm F}}

\newcommand{\FP}{{\rm FP}}
\newcommand{\FPinfty}{{\rm FP}\sb \infty}

\definecolor{cof}{RGB}{219,144,71}
\definecolor{pur}{RGB}{186,146,162}
\definecolor{greeo}{RGB}{91,173,69}
\definecolor{greet}{RGB}{52,111,72}

\usepackage{amssymb}
\usepackage{amsgen}
\usepackage{amsmath}
\usepackage{amsthm}
\usepackage{mathrsfs}
\usepackage{cite}
\usepackage{amsfonts}

\newcommand{\inv}{^{-1}}
\newcommand{\p}{\varphi}

\newcommand{\ov}[1]{\ensuremath{\overline {#1}}}

\newcommand{\Irr}{\mathop{\mathrm{Irr}}\nolimits}

\newtheorem{Thm}{Theorem}[section]
\newtheorem{Prop}[Thm]{Proposition}

\newtheorem{Lemma}[Thm]{Lemma}
{\theoremstyle{definition}
}
{\theoremstyle{remark}
\newtheorem{Rmk}[Thm]{Remark}}
{\theoremstyle{remark}
\newtheorem{remark}[Thm]{Remark}}
\newtheorem{Cor}[Thm]{Corollary}

{\theoremstyle{remark}
\newtheorem{Example}[Thm]{Example}}

\newtheorem{lem}[Thm]{Lemma}

{\theoremstyle{remark}
}
{\theoremstyle{remark}
}
{\theoremstyle{remark}
}

\numberwithin{equation}{section}

\usepackage{graphicx}

\title[Topological finiteness properties]{
Topological finiteness properties of monoids \\ Part 2: special monoids, one-relator monoids,  \\ amalgamated free products, and HNN extensions
}

\keywords{Equivariant CW-complex,
homological finiteness property,
classifying space,
cohomological dimension,
Hochschild cohomological dimension,
geometric dimension,
Bass--Serre tree,
monoid,
special monoid,
one-relator monoid,
free product with amalgamation,
HNN extension.
\\
\indent
This work was supported by the
EPSRC grants 
EP/N033353/1 `Special inverse monoids: subgroups, structure, geometry, rewriting systems and the word problem' and 
EP/V032003/1 `Algorithmic, topological and geometric aspects of infinite groups, monoids and inverse semigroups'. The second author was supported by  NSA MSP \#H98230-16-1-0047 and a PSC-CUNY award}

\begin{document}
\maketitle

\begin{center}
ROBERT D. GRAY
\footnote{School of Mathematics, University of East Anglia, Norwich NR4 7TJ, England.
Email \texttt{Robert.D.Gray@uea.ac.uk}.
}
and
BENJAMIN STEINBERG\footnote{
Department of Mathematics, City College of New York, Convent Avenue at 138th Street, New York, New York 10031,  USA.
Email \texttt{bsteinberg@ccny.cuny.edu}.}
\\
\end{center}

\begin{abstract}
We show how topological methods developed in a previous article can be applied to prove new results about topological and homological finiteness properties of monoids.
A monoid presentation is called special if the right-hand side of each relation is equal to $1$.
We prove results which relate the finiteness properties of a monoid defined by a special presentation with those of its group of units.
Specifically we show that the monoid inherits the finiteness properties $\F_n$ and $\FP_n$ from its group of units.
We also obtain results which relate the geometric and cohomological dimensions of such a monoid to those of its group of units.
We apply these results to prove a Lyndon's Identity Theorem for one-relator monoids of the form $\lb A \mid r=1 \rb$. 
In particular we show that all such monoids are of type $\F_{\infty}$ (and $\FP_{\infty}$), and that when $r$ is not a proper power, then the monoid has geometric and cohomological dimension at most $2$.   
The first of these results, resolves an important case of a question of Kobayashi from 2000 on homological finiteness properties of one-relator monoids.
We also show how our topological approach can be used to prove results about the closure properties of 
various homological and topological finiteness properties for 
amalgamated free products and HNN-extensions of monoids. 
To prove these results we introduce new methods for constructing equivariant
classifying spaces for monoids, as well as developing a Bass--Serre theory for free constructions of monoids.
 
\end{abstract}

\section{Introduction}
Topological methods play an important role in the modern study of infinite discrete groups.
Recall that an Eilenberg--Mac Lane complex of type $K(G,1)$ is an aspherical CW complex with fundamental group $G$.
For any group $G$ a $K(G,1)$ complex exists, and it is unique up to homotopy equivalence.
While the existence of such spaces is elementary, it is often a much harder problem to find a $K(G,1)$ complex which is suitably `nice' to be used for doing calculations.
This is important if one wants to compute homology and cohomology groups.
This is part of the motivation for the study of higher order topological finiteness properties of groups, a topic which goes back to pioneering work of Wall~\cite{Wall1965} and Serre~\cite{Serre1971}.
We recall that a group is of type $\F_n$ if there is a $K(G,1)$-complex with a finite $n$-skeleton.
The property $\F_1$ is equivalent to finite generation, while a group is of type $\F_2$ if and only if it is finitely presented, so $\F_n$ gives a natural higher dimensional analogue of these two fundamental finiteness properties.
The geometric dimension of a group $G$, denoted $\mathrm{gd}(G)$, is the minimum dimension of a $K(G,1)$ complex.  
The topological finiteness property $\F_n$ and geometric dimension correspond, respectively, to the homological finiteness property $\FP_n$ and the cohomological dimension of the group.
The study of topological and homological finiteness properties is an active area of research. 
We refer the reader to~\cite[Chapter~8]{BrownCohomologyBook},~\cite[Chapters~6-9]{GeogheganBook} and~\cite{Brown2010} for more background on this topic.

The homological finiteness properties $\FP_n$ and cohomological dimension have also been extensively studied more generally for monoids. 
One major motivation for studying homological finiteness properties of monoids comes from important connections with the theory of 
rewriting systems, and the word problem for finitely presented monoids. 
It is well known that there are finitely presented monoids with undecidable word problem.
Given that the word problem is undecidable in general, a central theme running through the development of geometric and combinatorial group and monoid theory has been to identify and study classes of finitely presented monoids all of whose members have solvable word problem.
 A finite complete rewriting system is a finite presentation for a monoid of a particular form (both confluent and Noetherian) which gives a solution of the word problem for the monoid; see~\cite{BookAndOtto}.
Complete rewriting systems are also of interest because of their close connection with 
the theory of Gr\"{o}bner--Shirshov bases; see~\cite{Ufnarovski1998}.
The connection between complete rewriting systems and homological finiteness properties is given by the Anick-Groves-Squier Theorem  which shows that a monoid that admits such a presentation must be of type 
$\FPinfty$; see \cite{Anick1986, Squier1987} and \cite{Brown1989}.
The property $\FP_n$ for monoids also arises in the study of Bieri--Neumann--Strebel--Renz invariants of groups; see~\cite{Bieri1987}.

A number of other interesting homological and  homotopical finiteness properties have been studied in relation to monoids defined by complete rewriting systems; see~\cite{AlonsoHermiller2003, 
Pride2005, Guiraud2012}.
The cohomological dimension of monoids has also received attention in the literature; see for example~\cite{Cheng1980, 
Guba1998, 
Margolis2015}.  
In fact, for monoids these properties depend on whether one works with left $\mathbb{Z}M$-modules or right $\mathbb{Z}M$-modules, giving rise to the notions of both left- and right-$\FP_n$, and left and right cohomological dimension.
In general these are independent of each other; see~\cite{Cohen1992, Guba1998, Pride2006}.
Working with bimodules resolutions of the $(\ZM, \ZM)$-bimodule $\ZM$ one obtains the notion bi-$\FPn$ introduced and studied in~\cite{KobayashiOtto2001}.
This property is of interest from the point of view of Hochschild cohomology, which is the standard notion of cohomology for rings; see~\cite{HochschildCoh, Mitchell1972}.
For more background on the study of homological finiteness properties in monoid theory, and the connections with the theory of string rewriting systems, see~\cite{Brown1989,Cohen1997,Otto1997}.

While homological finiteness properties of monoids have been extensively studied, in contrast, until recently there was no corresponding theory of topological finiteness properties of monoids.
The results in this paper are part of a research programme  of the authors, initiated in \cite{GraySteinberg1}, aimed at developing such a theory.
A central theme of this work is that the topological approach allows for less technical, and more conceptual, proofs than had previously been possible using only algebraic means.
Other recent results in the literature where topological methods have been usefully applied in the study of monoids include, e.g. \cite{
Brittenham2009, Meakin2015, Margolis2017}. 

This paper is the sequel to the article~\cite{GraySteinberg1} where we set out the foundations of $M$-equivariant homotopy theory for monoids acting on CW complexes,  and the corresponding study of topological finiteness properties of monoids.
In that paper we introduced the notion of a left equivariant classifying space for a monoid, which is a contractible projective $M$-CW complex. 
A left equivariant classifying space always exists, for any monoid $M$, and it is unique up to $M$-homotopy equivalence.
We then define the corresponding finiteness conditions left-$\F_n$ and left geometric dimension in the obvious natural way in terms of the existence of a left equivariant classifying space satisfying appropriate finiteness properties. 
It follows easily from the definitions that left-$\F_n$ implies left-$\FP_n$, and that the left geometric dimension is an upper bound on the left cohomological dimension of the monoid.
There are obvious dual definitions and statements working with right actions.
We also developed  a two-sided analogue of this theory in~\cite{GraySteinberg1}, with two-sided $M$ actions, defining the notion of a bi-equivariant classifying space for a monoid, and the resulting finiteness properties bi-$\F_n$ and geometric dimension.
It follows from the definitions that bi-$\F_n$ implies bi-$\FP_n$ (in the sense of~\cite{KobayashiOtto2001}) and that the geometric dimension is an upper bound for the Hochschild cohomological dimension.
See Section~\ref{sec_prelims} below for full details and formal definitions of all of these notions. 

The aim of this paper is to apply the ideas and results from~\cite{GraySteinberg1} to solve some open problems concerning homological finiteness properties of monoids that seemed resistant to algebraic techniques.
Let us begin with some history.
An important open problem is whether every one-relator monoid has decidable word problem.
While the question is open in general, it has been solved in a number of special cases; see Adjan~\cite{Adjan1966} and Adjan and 
Oganesyan~\cite{Adyan1987}.
 Related to this is another open question which asks whether every one-relator monoid admits a presentation by a finite complete rewriting system.
Of course, a positive answer to this question would imply a positive solution to the word problem.
In light of the Anick-Groves-Squier Theorem which states that monoids which admit finite complete presentations are of type right- and left-$\FP_\infty$, it is natural to ask whether all one-relator monoids are of type $\FP_\infty$.
This question was posed by Kobayashi in~\cite[Problem~1]{Kobayashi2000}. 
The question is also natural given the fact that all one-relator groups are all of type $\FP_\infty$, as a consequence of Lyndon's Identity Theorem  for one-relator groups; see Lyndon \cite{Lyndon1950}.

The first positive result concerning the word problem for one-relator monoids dealt with the case of, so-called, special one-relator monoids~\cite{Adjan1966}.
A \emph{special} monoid is one defined by a finite presentation of the form $\langle A\mid w_1=1,\ldots,w_k=1\rangle$.
They were first studied in the sixties by Adjan~\cite{Adjan1966} and Makanin~\cite{Makanin66}.
 Adjan proved that the group of units of a one-relator special monoid is a one-relator group and reduced the word problem of the monoid to that of the group, which has a decidable word problem by Magnus's theorem~\cite{LyndonAndSchupp}.
 Makanin proved more generally that the group of units of a $k$-relator special monoid is a $k$-relator group and reduced the word problem of the monoid to that of the group.
See~\cite{Zhang} for a modern approach to these results.
Thus there is a much closer connection for special monoids between the group of units and the monoid than is customary.

One of the main results of this paper is that if $M = \langle A\mid w_1=1,\ldots,w_k=1\rangle$, and if $G$ is the group of units of $M$, then   if $G$ is of type $\FPn$ with $1\leq n\leq \infty$, then $M$ is also of type left- and right-$\FPn$.
Moreover, we prove that both the left and right cohomological dimensions of $M$ are bounded below by $\mathop{\mathrm{cd}} G$, and are bounded above by $\max\{2,\mathop{\mathrm{cd}} G\}$.
We shall also prove the topological analogues of these results, obtaining the corresponding statements with right and left-$\F_n$ and geometric dimension.
These results are obtained by proving new results about the geometry of Cayley digraphs of special monoids, including the observation that the quotient of the Cayley digraph by its strongly connected components is a regular rooted tree on which the monoid acts by simplicial maps.
  We use this to show how one can construct a left equivariant classifying space for a special monoid from an equivariant classifying space for its group of units.

We shall then go on to apply these results to prove a  Lyndon's Identity Theorem \cite{Lyndon1950} for one-relator monoids of the form $\langle A\mid w=1\rangle$. 
Specifically, we show that our results can be applied to 
construct equivariant classifying spaces for one-relator monoids of 
this form,
which have finitely many orbits of cells in each dimension, and have dimension at most $2$ unless the monoid has torsion.  
We apply this to give a positive answer to Kobayashi's question \cite[Problem~1]{Kobayashi2000} on homological finiteness properties of one-relator monoids, in the case of one-relator monoids of the form $\langle A\mid w=1\rangle$, by proving that all such monoids are of type left- and right-$\F_\infty$ and $\FP_\infty$.
We also show that if $M=\langle A\mid w=1\rangle$ with $w$ not a proper power then the left and right cohomological dimension of $M$ are bounded above by $2$, and if $w$ is a proper power then they are both equal to $\infty$.
The analogous topological result for the left and right geometric dimension of a one-relator special monoid is also obtained.
In fact, it will follow from our results that when $w$ is not a proper power then the Cayley complex of the one-relator monoid $M$ is an equivariant classifying space for $M$ of dimension at most $2$.   
This is the analogue, for one-relator special monoids, of the fact that the presentation complex of a torsion-free one-relator group is aspherical and is thus a $K(G,1)$ complex for the group of dimension at most $2$; see \cite{Cockcroft1954, DyerVasquez1973}.  
These results on special monoids, and one-relator monoids, will be given in Section~\ref{sec_special}.

The results we obtain in this paper for special one-relator monoids form an important infinite family of base cases for the main result in our article \cite{GraySteinberg3} where we prove a Lyndon's Identity Theorem for arbitrary one-relator monoids $\langle A\mid u=v \rangle$.  Applying this result, in \cite{GraySteinberg3} we give a positive answer to Kobayashi's question by showing that every one-relator monoid $\lb A \mid u=v \rb$ is of type left- and right-$\FP_{\infty}$.

In Section~\ref{sec_amalg} below we prove several new results about the preservation of topological and homological finiteness properties for amalgamated free products of monoids.
Monoid amalgamated products are far more complicated than group ones.
For example, an amalgamated free product of finite monoids can have an undecidable word problem, and the factors do not necessarily embed, or intersect, in the base monoid; see~\cite{Sapir2000}.
In particular there are no normal form results at our disposal when working with monoid amalgamated free products.
We give a method for constructing  an equivariant classifying spaces for an amalgamated free product of monoids $L = M_1 \ast_W M_2$ from equivariant classifying spaces of the monoids $M_1$, $M_2$ and $W$.
To do this, we use homological ideas of Dicks~\cite{Dicks80} on derivations to construct a Bass--Serre tree $T$ for the amalgam $L$.
 We also develop an analogous theory in the two-sided case. 
These constructions are used to prove several results about the closure properties of $\F_n$, $\FP_n$, and geometric and cohomological dimension.

Finally, in Section~\ref{sec_HNNOttoPride} we consider HNN extensions construction for monoids, in the sense of Otto and Pride \cite{Pride2004}, and those defined by Howie \cite{Howie1963}. 
As in the case of amalgamated free products, we give constructions of equivariant classifying spaces, and apply these to deduce results about the closure properties of topological and homological finiteness properties. 
This also involves constructing appropriate Bass--Serre trees.  
As special cases of our results we recover generalisations of a number of results of Otto and Pride from~\cite{Pride2004} and~\cite{Pride2005}.

\section{Preliminaries}
\label{sec_prelims}

In this section we recall some of the relevant background from~\cite{GraySteinberg1} needed for the rest of the article. For full details, and proofs of the statements made here we refer the reader to~\cite[Sections~2-4]{GraySteinberg1}. For additional general background on algebraic topology, and topological methods in group theory, we refer the reader to~\cite{May1999} and~\cite{GeogheganBook}.

\subsection{The category of $M$-sets}

Let $M$ be a monoid. A \emph{left $M$-set} consists of a set $X$ and a mapping
$M \times X \rightarrow X$ written $(m,x) \mapsto mx$ called a \emph{left
action}, such that $1x=x$ and $m(nx) = (mn)x$ for all $m,n \in M$ and
$x \in X$. Right $M$-sets are defined dually, they are the same thing
as left $M^{op}$-sets, where $M^{op}$ is the \emph{opposite} of the monoid $M$
which is the monoid with the same underlying set $M$ and multiplication given by
$x \cdot y = yx$.  A \emph{bi-$M$-set} is an $M\times M^{op}$-set.
A mapping $f\colon X \rightarrow Y$ between $M$-sets is \emph{$M$-equivariant} if
$f(mx) = mf(x)$ for all $x \in X$, $m \in M$, and $M$-sets together with $M$-equivariant
mappings form a category.

If $X$ is an $M$-set and $A\subseteq X$, then $A$ is said to be a \emph{free basis for $X$} if and only if each element of $X$ can be uniquely expressed as $ma$ with $m\in M$ and $a\in A$. The free left $M$-set on $A$ exists and can be realised as the set $M \times A$ with action $m(m',a) = (mm',a)$. Note that if $G$ is a  group, then a left $G$-set $X$ is free if and only if $G$ acts freely on $X$, that is, each element of $X$ has trivial stabilizer.  In this case, any set of orbit representatives is a basis. An  $M$-set $P$ is \emph{projective} if any $M$-equivariant surjective mapping $f\colon X\to P$ has an $M$-equivariant section $s\colon P\to X$ with $f\circ s=1_P$.  Every free $M$-set is projective, and an $M$-set is projective if and only if it is a retract of a free one. Each projective $M$-set $P$ is isomorphic to an $M$-set of the form $\coprod_{a\in A} Me_a$ (disjoint union, which is the coproduct in the category of $M$-sets) with $e_a\in E(M)$, where $E(M)$ denotes the set of idempotents of the monoid $M$. 
In particular, projective $G$-sets are the same thing as free $G$-sets for a group $G$.

If $A$ is a right $M$-set and $B$ is a left $M$-set, then $A\otimes_M B$ is the quotient of $A\times B$ by the least equivalence relation $\sim$ such that $(am,b)\sim (a,mb)$ for all  $a\in A$, $b\in B$ and $m\in M$.  We write $a\otimes b$ for the class of $(a,b)$ and note that the mapping $(a,b)\mapsto a\otimes b$ is universal for mappings $f\colon A\times B\to X$ with $X$ a set and $f(am,b)=f(a,mb)$.  If $M$ happens to be a group, then $M$ acts on $A\times B$ via $m(a,b)=(am^{-1},mb)$ and $A\otimes_M B$ is just the set of orbits of this action. The tensor product $A\otimes_M()$ preserves all colimits because it is a left adjoint to the functor $X\mapsto X^A$.

If $B$ is a left $M$-set there is a natural preorder relation $\leq$ on $B$ where $x \leq y$ if and only if $Mx \subseteq My$. We write $x \approx y$ if there is a sequence $z_1, z_2, \ldots, z_n$ of elements of $B$ such that for each $0 \leq i \leq n-1$ either $z_i \leq z_{i+1}$ or $z_i \geq z_{i+1}$. This is clearly an equivalence relation and we call the $\approx$-classes of $B$ the \emph{weak orbits} of the $M$-set. This corresponds to the notion of the weakly connected components in a directed graph. If $B$ is a right $M$-set then we use $B/M$ to denote the set of weak orbits of the $M$-set while if $B$ is a left $M$-set we use $M\backslash B$ to denote the set of weak orbits. Note that if $1$ denotes the trivial right $M$-set and $B$ is a left $M$-set, then we have $M\backslash B=1\otimes_M B$. Let $M,N$ be monoids.  An \emph{$M$-$N$-biset} is an $M\times N^{op}$-set. If $A$ is an $M$-$N$-biset and $B$ is a left $N$-set, then the equivalence relation defining $A\otimes_N B$ is left $M$-invariant and so $A\otimes_N B$ is a left $M$-set with action $m(a\otimes b) = ma\otimes b$.

\subsection{Projective $M$-CW complexes}
A \emph{left $M$-space} is a topological space $X$ with a continuous left action $M\times X\to X$ where $M$ has the discrete topology.  A right $M$-space is the same thing as a left $M^{op}$-space and a \emph{bi-$M$-space} is an $M\times M^{op}$-space.  Each $M$-set can be viewed as a discrete $M$-space.
Colimits in the category of $M$-spaces are formed by taking colimits in the category of spaces and observing that the result has a natural $M$-action.

Our main interest in this article will be in  $M$-spaces $X$ where $X$ is a CW complex.
Following~\cite{GraySteinberg1} we define a (projective) \emph{$M$-cell} of dimension $n$ to be an $M$-space of the form $Me\times B^n$ where $e\in E(M)$ is an idempotent and $B^n$ has the trivial action. In the special case $e=1$, we call it a \emph{free $M$-cell}.  We then define a projective $M$-CW complex in an inductive fashion by imitating the usual definition of a CW complex but by attaching $M$-cells $Me\times B^n$ via $M$-equivariant maps from $Me\times S^{n-1}$ to the $(n-1)$-skeleton. Formally, a \emph{projective (left) relative $M$-CW complex} is a pair $(X,A)$ of $M$-spaces such that $X=\varinjlim X_n$ with $i_n\colon X_n\to X_{n+1}$ inclusions, $X_{-1}=A$, $X_0 = P_0\cup A$ with $P_0$ a projective $M$-set and where $X_n$ is obtained as a pushout of $M$-spaces
\begin{equation}\label{eq:pushout}
\begin{tikzcd}P_n\times S^{n-1}\ar{r}\ar[d,hook] & X_{n-1}\ar[d,hook]\\ P_n\times B^n\ar{r} & X_n \end{tikzcd}
\end{equation}
with $P_n$ a projective $M$-set and $B^n$ having a trivial $M$-action for $n\geq 1$.  The set $X_n$ is the \emph{$n$-skeleton} of $X$ and if $X_n=X$ and $P_n\neq \emptyset$, then $X$ is said to have \emph{dimension} $n$.
Since $P_n$ is isomorphic to a coproduct of $M$-sets of the form $Me$ with $e\in E(M)$, we are indeed attaching $M$-cells at each step. If $A=\emptyset$, we call $X$ a \emph{projective $M$-CW complex}. Note that a projective $M$-CW complex is a CW complex and the $M$-action is cellular (in fact, takes $n$-cells to $n$-cells).  We can define projective right $M$-CW complexes and projective bi-$M$-CW complexes by replacing $M$ with $M^{op}$ and $M\times M^{op}$, respectively. We say that $X$ is a \emph{free $M$-CW complex} if each $P_n$ is a free $M$-set.
A projective $M$-CW complex $X$ is of \emph{$M$-finite type} if $P_n$ is a finitely generated projective $M$-set for each $n$, and we say that $X$ is \emph{$M$-finite} if it is finite dimensional and of $M$-finite type (i.e., $X$ is constructed from finitely many $M$-cells).
The degree $n$ component of the cellular chain complex for
the projective $M$-CW complex
$X$ is isomorphic to $\mathbb ZP_n$ as a $\mathbb ZM$-module, and hence is projective. 

A \emph{projective $M$-CW subcomplex} of $X$ is an $M$-invariant subcomplex $A\subseteq X$ which is a union of $M$-cells of $X$.
If $X$ is a projective $M$-CW complex then so is $Y=X\times I$ where $I$ is given the trivial action.  If we retain the above notation, then $Y_0=X_0\times \partial I\cong X_0\coprod X_0$.  The $n$-cells for $n\geq 1$ are obtained from attaching $P_n\times B^n\times \partial I\cong (P_n\coprod P_n)\times B^n$ and $P_{n-1}\times B^{n-1}\times I$.  Notice that $X\times \partial I$ is a projective $M$-CW subcomplex of $X\times I$.
An \emph{$M$-homotopy} between $M$-equivariant continuous maps $f,g\colon X\to Y$ between $M$-spaces $X$ and $Y$ is an $M$-equivariant mapping $H\colon X\times I\to Y$ with $H(x,0)=f(x)$ and $H(x,1)=g(x)$ for $x\in X$ where $I$ is viewed as having the trivial $M$-action.  We write $f\simeq_M g$ in this case.  We say that $X,Y$ are \emph{$M$-homotopy equivalent}, written $X\simeq_M Y$, if there are $M$-equivariant continuous mappings (called \emph{$M$-homotopy equivalences}) $f\colon X\to Y$ and $g\colon Y\to X$ such that $gf\simeq_M 1_X$ and $fg\simeq_M 1_Y$.  
Every $M$-equivariant continuous mapping of projective $M$-CW complexes is $M$-homotopy equivalent to a cellular one. This is the \emph{cellular approximation theorem}; see \cite[Theorem 2.8]{GraySteinberg1}.

If $X$ is a left $M$-space and $A$ is a right $M$-set, then $A\otimes_M X$ is a topological space with the quotient topology.  The following base change result will be used frequently below.  
\begin{Prop}\label{c:base.change.cw}
\emph{\cite[Proposition~3.1 and Corollary~3.2]{GraySteinberg1}}
If $A$ is an $M$-$N$-biset that is projective (free) as an $M$-set and $X$ is a projective (free) $N$-CW complex, then $A\otimes_N X$ is a projective (free) $M$-CW complex.  If $A$ is in addition finitely generated as an $M$-set and $X$ is of $N$-finite type, then $A\otimes_N X$ is of $M$-finite type.  Moreover, $\dim A\otimes_N X=\dim X$.
\end{Prop}

\begin{Rmk}\label{r:tensor.with.free}
 We shall use the observation that if $X$ is a free right $M$-set on $A$, then $A$ is in bijection with $X/M$ and hence $X\cong X/M\times M$ as a right $M$-set where $M$ acts trivially on $X/M$. Hence if $Y$ is a projective $M$-CW complex, then $X\otimes_M Y\cong \coprod_A Y\cong X/M\times Y$ where $X/M$ has the discrete topology.  Moreover, these homeomorphisms come from isomorphisms of the CW structure.
\end{Rmk}

\subsection{Equivariant classifying spaces and topological finiteness properties for monoids}

A \emph{(left) equivariant classifying space} $X$ for a monoid $M$ is a projective $M$-CW complex which is contractible. A right equivariant classifying space for $M$ will be a left equivariant classifying space for $M^{op}$. 
In some cases, an equivariant classifying space for a monoid may be constructed using the Cayley digraph of the monoid as the $1$-skeleton.  Recall that if $M$ is a monoid and $A\subseteq M$, then the (right) \emph{Cayley digraph} $\Gamma(M,A)$ of $M$ with respect to $A$ is the graph with vertex set $M$ and with edges in bijection with $M\times A$ where the directed edge (arc) corresponding to $(m,a)$ starts at $m$ and ends at $ma$.  Note that $\Gamma(M,A)$ is a free $M$-graph and is $M$-finite if and only if $A$ is finite (see Section~\ref{sec_amalg} below for the definition of $M$-graph).

Equivariant classifying spaces of monoids are unique up to $M$-homotopy equivalence; see~\cite[Theorem~6.3 \& Corollary~6.5]{GraySteinberg1}.
The definition of equivariant classifying spaces for monoids leads naturally to the definitions of the following topological finiteness properties.
A monoid $M$ is of type \emph{left-$\F_n$} (for a non-negative integer $n$) if there is an equivariant classifying space $X$ for $M$ such that $X_n$ is $M$-finite, i.e., such that $M\backslash X$ has finite $n$-skeleton.  We say that $M$ is of type \emph{left-$\F_{\infty}$} if $M$ has an equivariant classifying space $X$ that is of $M$-finite type, i.e., $M\backslash X$ is of finite type. The monoid $M$ is defined to have type \emph{right-$\F_n$} if $M^{op}$ is of type left-$\F_n$ for $0\leq n\leq \infty$. The \emph{left geometric dimension} of $M$ is defined to be the minimum dimension of a left equivariant classifying space for $M$.  The right geometric dimension is defined dually.

The homological analogue of left-$\Fn$ is the finiteness property left-$\FPn$, where a monoid $M$ is said to be of type left-$\FPn$ if there is a projective resolution $P = (P_i)_{i \geq 0}$ of the trivial left $\ZM$-module $\Z$ such that $P_i$ is finitely generated for $i \leq n$.
There is a dual notion of right-$\FPn$, and we say a monoid is of type $\FPn$ if it is both of type left- and right-$\FPn$.
For any monoid $M$, if $M$ is of type left-$\F_n$ for some $0 \leq n \leq \infty$ then it is of type left-$\FP_n$. Indeed, if $X$ is an equivariant classifying space for $M$ then the augmented cellular chain complex of $X$ gives a projective $\ZM$-resolution of the trivial $\ZM$-module $\Z$ with the desired finiteness properties.
If $M$ is a monoid of type left-$\F_2$, then $M$ is of type left-$\F_n$ if and only if $M$ is of type left-$\FP_n$ for $0\leq n\leq \infty$. In particular, for finitely presented monoids the conditions left-$\F_n$ and left-$\FP_n$ are equivalent. In the special case that the monoid $M$ is a group, the definition of left-$\F_n$ above is easily seen to agree with the usual definition of $\F_n$ for groups.
The left geometric dimension is clearly an upper bound on the left cohomological dimension, denoted $\mathop{\mathrm{left \; cd}} M$, of a monoid $M$ where 
the left \emph{cohomological dimension} of $M$ is the shortest length of a projective resolution of the trivial left $\ZM$-module $\mathbb Z$.

To define the bilateral notion of a classifying space, first recall that $M$ is an $M\times M^{op}$-set via the action $(m_L,m_R)m=m_Lmm_R$.  We say that a projective $M\times M^{op}$-CW complex $X$ is a \emph{bi-equivariant classifying space for $M$} if $\pi_0(X)\cong M$ as an $M\times M^{op}$-set and each component of $X$ is contractible; equivalently, $X$ has an $M\times M^{op}$-equivariant homotopy equivalence to the discrete $M\times M^{op}$-set $M$. We can augment the cellular chain complex of $X$ via the canonical surjection $\varepsilon\colon C_0(X)\to H_0(X)\cong \mathbb Z\pi_0(X)\cong \mathbb ZM$. Since each component of $X$ is contractible, this gives a projective bimodule resolution of $\mathbb ZM$.
A bi-equivariant classifying space may be constructed for any monoid~\cite[Corollary~7.4]{GraySteinberg1}. 
As in the one-sided case, bi-equivariant classifying spaces are unique up to $M\times M^{op}$-homotopy equivalence; see~\cite[Theorem~7.2]{GraySteinberg1}.

A monoid $M$ is said to be of type \emph{bi-$\F_n$} if there is a bi-equivariant classifying space $X$ for $M$ such that $X_n$ is $M\times M^{op}$-finite, i.e., $M\backslash X/M$ has finite $n$-skeleton.  We say that $M$ is of type \emph{bi-$\F_{\infty}$} if $M$ has a bi-equivariant classifying space $X$ that is of $M\times M^{op}$-finite type, i.e., $M\backslash X/M$ is of finite type. We define the \emph{geometric dimension} of $M$ to be the minimum dimension of a bi-equivariant classifying space for $M$.
The homological analogue of bi-$\F_n$ is the property \emph{bi-$\FPn$} (in the sense of~\cite{KobayashiOtto2001}), where a monoid is said to be of type \emph{bi-$\FPn$}
if there is a projective resolution
\[
 \cdots \rightarrow P_1 \rightarrow
P_0 \rightarrow \mathbb ZM \rightarrow 0
\]
of the $(\mathbb ZM,\mathbb ZM)$-bimodule $\mathbb ZM$, where $P_0, P_1, \ldots, P_n$ are finitely generated projective $(\mathbb{Z}M, \mathbb{Z}M)$-bimodules. For $0\leq n\leq \infty$, if $M$ is of type bi-$\F_n$, then it is of type bi-$\FP_n$. 
If $M$ is of type bi-$\F_n$ for $0\leq n\leq\infty$, then $M$ is of type left-$\F_n$ and type right-$\F_n$. If $M$ is a monoid of type bi-$\F_2$, then $M$ is of type bi-$\F_n$ if and only if $M$ is of type bi-$\FP_n$ for $0\leq n\leq \infty$; see~\cite[Theorem 7.15]{GraySteinberg1}. In particular, for finitely presented monoids bi-$\F_n$ and bi-$\FP_n$ are equivalent.
The \emph{Hochschild cohomological dimension} of $M$, written $\dim M$, is the length of a shortest projective resolution of $\mathbb ZM$ as a $\mathbb Z[M\times M^{op}]$-module. The Hochschild cohomological dimension bounds both the left and right cohomological dimension and the geometric dimension bounds the Hochschild cohomological dimension. The geometric dimension also bounds both the left and right geometric dimensions because if $X$ is a bi-equivariant classifying space for $M$ of dimension $n$, then $X/M$ is an equivariant classifying space of dimension $n$. 

\subsection{A theorem of Brown}

We end this section by recalling a result of Brown which will be useful for proofs of results about homological finiteness properties of monoids. Unless otherwise stated, all modules considered here are left modules.
Let us say that a module $V$ over a (unital) ring $R$ is of type $\FP_n$ if it has a projective resolution that is finitely generated through degree $n$; this is equivalent to having a free resolution that is finitely generated through degree $n$; see~\cite[Proposition~4.3]{BrownCohomologyBook}.  We say that $V$ is of type $\FP_{\infty}$ if it has a projective (equivalently, free) resolution that is finitely generated in all degrees.  So a monoid is of type left $\FP_n$ if and only if the trivial left module is of type $\FP_n$. One says that $V$ has \emph{projective dimension} at most $d$ if it has a projective resolution of length $d$.  Note that the left cohomological dimension of a monoid is the projective dimension of the trivial left module.
 Notice also that both the class of modules of type $\FP_n$ and the class of modules having projective dimension at most $d$ are closed under direct sum.

The following is
lemma of K.~Brown~\cite{Brown1982}. Recall that a morphism of chain
complexes is a \emph{weak equivalence} if it induces an isomorphism on
homology.

\begin{Lemma}\emph{
\cite[Lemma~1.5]{Brown1982}
}\label{l:brown}
Let $R$ be a ring and $C=(C_i)$ a chain complex of (left) $R$-modules and, for each $i$, let $(P_{ij})_{j\geq 0}$ be a projective resolution of $C_i$.  Then one can find a chain complex $Q=(Q_n)$ with $Q_n = \bigoplus_{i+j=n} P_{ij}$ such that there is a weak equivalence $f\colon Q\to C$.
\end{Lemma}

\begin{Cor}\label{c:fp.resolved}
Suppose that $R$ is a ring and
\[C_n\longrightarrow C_{n-1}\longrightarrow\cdots\longrightarrow C_0\longrightarrow V\] is a partial resolution of an $R$-module $V$.
\begin{enumerate}
\item If $C_i$ is of type  $\FP_{n-i}$, for $0\leq i\leq n$, then $V$ is of type $\FP_n$.
\item Let $d\geq n$ and suppose that $C_n\to C_{n-1}$ is injective. If $C_i$ has a projective dimension of at most $d-i$, for $0\leq i\leq n$, then $V$ has a projective dimension at most $d$.
\end{enumerate}
\end{Cor}
\begin{proof}
To prove the first item, put $C=(C_i)$ and let $(P_{ij})_{j\geq 0}$ be a projective resolution of $C_i$ by finitely generated projectives that is finitely generated through degree $n-i$.  Then the chain complex $Q$ from Lemma~\ref{l:brown} is a complex of projectives with $Q_k$ finitely generated, for $0\leq k\leq n$, with $H_0(Q)\cong H_0(C)=V$ and $H_q(Q)\cong H_q(C)=0$ for $0<q<n$.  Thus if we augment
\[Q_n\longrightarrow Q_{n-1}\longrightarrow \cdots \longrightarrow Q_0\]
by the natural epimorphism $Q_0\to H_0(Q)\cong V$, we obtain a partial projective resolution of $V$ of length $n$ by finitely generated projectives.

For the second item, again let $C=(C_i)$ and let $(P_{ij})_{j\geq 0}$ be a projective resolution of $C_i$ of length at most $d-i$.  Then the chain complex $Q$ from Lemma~\ref{l:brown} is a complex of projectives of length at most $d$ with $H_0(Q)\cong H_0(C)\cong V$ and $H_q(Q)=H_q(C)=0$ for $q>0$.  Thus if we augment $Q$ by the canonical epimorphism $Q_0\to H_0(Q)\cong V$, we obtain a projective resolution of $V$ of length at most $d$.
\end{proof}

Next we show that projective dimension and $\FP_n$ are stable under flat base extension.

\begin{Lemma}\label{l:flat.base}
Suppose that $\p\colon R\to S$ is a ring homomorphism and that $S$ is flat as a right $R$-module.  Let $V$ be a left $R$-module.
\begin{enumerate}
  \item If $V$ is of type $\FP_n$, then $S\otimes_R V$ is of type $\FP_n$ as an $S$-module.
  \item If $V$ has projective dimension at most $d$, then $S\otimes_R V$ has projective dimension at most $d$ over $S$.
\end{enumerate}
\end{Lemma}
\begin{proof}
Since $S\otimes_R R\cong S$ and tensor products preserve direct sums and retracts, it follows that if $P$ is a (finitely generated) projective $R$-module, then $S\otimes_R P$ is a (finitely generated) projective $S$-module.  If $(P_i)$ is a projective resolution of $V$, then by flatness of $S$ and the preceding observation, we obtain that $(S\otimes_R P_i)$ is a projective resolution of $S\otimes_R V$ with $S\otimes_R P_i$ finitely generated whenever $P_i$ is.  The result follows.
\end{proof}

A typical way to apply Corollary~\ref{c:fp.resolved} in order to prove that a monoid $M$ is of type $\FPn$ is to find an action of $M$ by cellular mappings on a contractible CW complex $X$ such that the $i^{th}$-cellular chain group  $C_i(X)$ is of type $\FP_{n-i}$ as a $\mathbb ZM$-module for $0\leq i\leq n$.

\section{Special monoids and one-relator monoids}
\label{sec_special}

Let $M$ be the monoid defined by the finite presentation $\langle A\mid w_1=1,\ldots,w_k=1\rangle$. Presentations of this form are called \emph{special}, and monoids which admit such presentations are called \emph{special monoids}.  Special presentations were first studied by Adjan~\cite{Adjan1966} and Makanin~\cite{Makanin66}. The main aim of this section is to prove some results which relate the topological and homological finiteness properties of special monoids to the corresponding properties holding in their group of units. By specialising to the case of one-relator monoids and combining with results of Adjan~\cite{Adjan1966} and Lyndon~\cite{Lyndon1950}
we then obtain a result characterising homological and cohomological finiteness properties of special one-relator monoids.  These results answer an important case of the open problem of Kobayashi~\cite{Kobayashi2000} which asks whether all one-relator monoids are of type right and left-$\FPinfty$. As discussed in the introduction to this paper, additional motivation for this question comes from its connection to the question of whether one-relator monoids admit presentations by finite complete rewriting systems which, in turn, relates to the longstanding open problem of whether such monoids have decidable word problem.

For rewriting systems we follow~\cite[Chapter~12]{HoltBook}. We recall some basic definitions and notation here.
Let $A$ be a non-empty set, known as an alphabet, and let $A^*$ denote the free monoid of all words over $A$. 
If $w = a_1 a_2 \ldots a_n \in A^*$, with $a_i \in A$ for $1 \leq i \leq n$, then we write $|w| = n$ and call this the \emph{length} of the word $w$. 
A \emph{rewriting system} $\rel$ over $A$ is a subset of $A^* \times A^*$. The pair $\lb A \mid \rel \rb$ is called a \emph{monoid presentation}. The elements of $\rel$ are called \emph{rewrite rules}. For words $u,v \in A^*$ we write $u \rightarrow_\rel v$ if there are words $\alpha, \beta \in A^*$ and a rewrite rule $(l,r)$ in $\rel$ such that $u = \alpha l \beta$ and $v = \alpha r \beta$. We use $\rightarrow_\rel^*$ to denote the reflexive transitive closure of $\rightarrow_\rel$, while $\leftrightarrow_\rel^*$ denotes the symmetric closure of $\rightarrow_\rel^*$. The relation $\leftrightarrow_\rel^*$ defines a congruence on $A^*$ and the quotient $A^* /  \leftrightarrow_\rel^*$ is called the \emph{monoid defined by the presentation} $\lb A \mid \rel \rb$. For any word $w \in A^*$ we use $[w]_\rel$ to denote the $\leftrightarrow_\rel^*$-class of the word $w$. So for words $u, v \in A^*$ when we write $u=v$ it means that $u$ and $v$ are equal as words in $A^*$, while $[u]_\rel=[v]_\rel$ means that $u$ and $v$ represent the same element of the monoid defined by the presentation.  We also sometimes write $u=_{\rel} v$ to mean the $[u]_\rel=[v]_\rel$. When the set of rewrite rules with respect to which we are working with is clear from context, we shall often omit the subscript $\rel$ and simply write $[u]$, $\rightarrow$, $\rightarrow^*$ and $\leftrightarrow^*$.

A word $u$ is called \emph{irreducible} if no rewrite rule can be applied to it, that is, there is no word $v$ such that $u \rightarrow v$. We use $\Irr(\rel)$ to denote the set of irreducible words of the system $\rel$. The rewriting system $\rel$ is \emph{Noetherian} if there is no infinite chain of words $u_i \in A^*$ with $u_i \rightarrow u_{i+1}$ for all $i \geq 1$. The system is \emph{confluent} if whenever $u \rightarrow^* u_1$ and $u \rightarrow^* u_2$ there is a word $v \in A^*$ such that $u_1 \rightarrow^* v$ and $u_2 \rightarrow^* v$. A rewriting system that is both Noetherian and confluent is called \emph{complete}. If $\rel$ is a complete rewriting system then each $\leftrightarrow^*$ equivalence class contains a unique irreducible word. Thus in this situation, $\Irr(\rel)$ provides a set of normal forms for the elements of the monoid defined by the presentation $\lb A \mid \rel \rb$.

Let $M=\langle A\mid w_1=1,\ldots,w_k=1\rangle = \langle A \mid T \rangle$  be the finitely presented special monoid defined above. 
The symbol $M$ will be used to denote this monoid for the remainder of this section. 
We call $w_1, w_2, \ldots, w_k$ the  \emph{defining relators} of this presentation. Let $\Gamma(M,A)$ denote the right Cayley graph of $M$ with respect to $A$. The strongly connected components of $\Gamma(M,A)$ are called the \emph{Sch\"{u}tzenberger graphs} of $M$. Here we say that two vertices $u$ and $v$ of a directed graph belong to the same strongly connected component if and only if there is a directed path from $u$ to $v$, and also a directed path from $v$ to $u$. Our aim is to prove that any two Sch\"{u}tzenberger graphs of $M$ are isomorphic to each other and that, modulo the Sch\"{u}tzenberger graphs, the Cayley graph of $M$ has a tree-like structure. We begin by summarising some results of Zhang~\cite{Zhang} on special monoids that will be used extensively below.

Let $G$ be the group of units of $M$.
By~\cite[Theorem~3.7]{Zhang}, we have that $G$ has a group
presentation with $k$ defining relations. Let $R$ be the submonoid of
right invertible elements. Then $R$ is isomorphic to a free product of
$G$ with a finitely generated free monoid
by~\cite[Theorem~4.4]{Zhang}.

In more detail, we say that a word $u \in A^*$ is invertible
if $[u] \in M$ is invertible.
Let
$u \in A^+$ be a non-empty invertible word.
We say that the invertible word $u$ is
\emph{indecomposable} if no non-empty proper prefix of $u$ is invertible.
Every non-empty invertible word $v$ has a unique decomposition
$v = v_1 v_2 \ldots v_l$ where each $v_i$ is indecomposable.
To obtain this decomposition, first write $v = v_1 u_1$ where $v_1$ is the shortest non-empty invertible prefix of $v$. Since $v$ and $v_1$ are invertible it follows that $u_1$ is invertible. If $u_1$ is non-empty we repeat this process writing $u_1 = v_2 u_2$ where $v_2$ is the shortest non-empty invertible prefix of $u_1$. Continuing in this way gives the decomposition $v = v_1 v_2 \ldots v_l$. It is unique since if $v_1' v_2' \ldots v_k'$ were some other such decomposition then
$v_1 v_2 \ldots v_l = v_1' v_2' \ldots v_k'$,
 neither $v_1$ nor $v_1'$ can be a proper prefix of the other, hence $v_1 = v_1'$, and then inductively we see that $v_i = v_i'$ for all $i$.
We call $u \in A^+$ a \emph{minimal invertible word} if it is indecomposable and invertible and the length of $u$ does not exceed the length of any of the relators in $T$.
Each relation word $w_i$ in $T$ represents the identity of $M$ and
thus is invertible. Therefore each relation word $w_i$ has a unique
decomposition $ w_i = w_{i,1} w_{i,2} \ldots w_{i, n_i} $ into
indecomposable  invertible words. The words $w_{i,j}$ for $1 \leq i \leq n$,
$1 \leq j \leq n_j$ are called the \emph{minimal factors} of the
relators of the presentation. Each minimal factor is clearly a
minimal invertible word.

Let $\Delta$ be the set of all minimal invertible words
$\delta \in A^*$ such that $\delta$ is equal in $M$ to at least one of
the minimal factors $w_{i,j}$ of the relators. Clearly $\Delta$ is a
finite set of words over $A$. It is also immediate from the definition
that $\Delta$ contains in particular all of the minimal factors $w_{i,j}$ of
the relators. It is also a consequence of the definitions that no
non-empty proper prefix of a word from $\Delta$ can be equal to a
non-empty proper suffix of a word from $\Delta$. On the other hand, a
word from $\Delta$ can, in general, arise as a subword of a word from
$\Delta$ (and there are examples where this happens).  It also follows from the definitions that $\Delta$ is a prefix code, meaning that no word from $\Delta$ is a prefix of any other word from $\Delta$.  It follows that $\Delta$ freely generates a free submonoid of $A^*$.

The elements represented by the words from $\Delta$ give a finite
generating set for the group of units $G$ of the monoid $M$. Indeed, it may be shown that every indecomposable invertible word $v$ is equal in $M$ to some word from $\Delta$; see~\cite[Lemma~3.4]{Zhang}, and every invertible word can be written as a product of indecomposable invertible words.

A finite presentation for the group of units $G$ of $M$, with respect to the finite generating set $\Delta$ may be constructed in the following way. We partition the finite set of words $\Delta$ as the disjoint union $ \Delta = \Delta_1 \cup \Delta_2 \cup \ldots \cup \Delta_m $
of non-empty sets
where two words belong to the same set $\Delta_j$ if and only if they represent the same element of the monoid $M$. Note that two distinct factors $w_{i,j}$ could well represent the same element of $M$ even if they are not equal as words. Set $B = \{b_1, b_2, \ldots, b_m\}$ and define a map $\phi$ from $\Delta$ to $B$ which maps every word from the set $\Delta_j$ to the letter $b_j$. Extend this to a surjective homomorphism $\phi\colon  \Delta^* \rightarrow B^*$.
Note that for any word $v \in A^*$, if $v \in \Delta^*$ then as observed above $v$ has a unique decomposition $v = v_1 v_2 \ldots v_l$ where each $v_i \in \Delta^*$ and thus the mapping $\phi$ is well-defined on the subset $\Delta^*$ of $A^*$.
Let $T_0$ be the rewriting system over the alphabet $B$ given by applying $\phi$ to each of the relators from the presentation $\lb A \mid T \rb$ (recall that each $w_j\in \Delta^*$) to obtain
\[
T_0  =  \{
(s,1):
\mbox{$s$ is some cyclic permutation of some $\phi(w_j)$}\}.
\]
This means for each relator $w_j$ from $T$, we decompose $w_j$ into its minimal factors, then read the factors recording the sets $\Delta_i$ to which each of them belongs, and then write down the corresponding word over $B$, and all of its cyclic conjugates.

\begin{Thm} \emph{~\cite[Theorem~3.7]{Zhang} } Let $M$ be the monoid
  defined by a finite special presentation $\lb A \mid T \rb$. Then
  $\lb B \mid T_0 \rb$ is a finite monoid presentation for the group
  of units $G$ of $M$. \end{Thm}

It follows that $\langle B\mid \phi(w_1)=1,\ldots, \phi(w_k)=1\rangle$ is a group presentation for the group of units of $M$ with the same number of defining relations as the presentation of $M$.

Choose and fix some order on the finite alphabet $A$, and for words
$x,y \in A^*$ write $x < y$ if $x$ precedes $y$ in the resulting
shortlex ordering~\cite[Definition~2.60]{HoltBook}. Now define a rewriting system
$S = S(T)$ over $A$ as follows:
\[
S = \{
(u,v) \mid
u, v \in \Delta^*\colon
\phi(u) =_{T_0} \phi(v) \ \& \ u > v \}.
\]
In fact, it follows from the results of Zhang that the condition $\phi(u) =_{T_0} \phi(v)$ is
equivalent to saying that $u =_T v$, i.e. that $u$ and $v$ represent
the same element of the group of units of the monoid $M$. So the
condition $\phi(u) =_{T_0} \phi(v)$ could be replaced by the condition
$u =_T v$ in the definition of $S$.

\begin{Thm}
\emph{
\cite[Proposition~3.2]{Zhang}
}
The infinite presentation $\lb A \mid S \rb$ is
Noetherian, confluent and defines the monoid $M$.  In fact, the rewriting systems $T$ and $S=S(T)$ are equivalent, that is, $\leftrightarrow_S^* = \leftrightarrow_T^*$.
\end{Thm}

We shall prove statements about $M$ by working with the irreducible
words $\Irr(S)$ associated with this infinite complete rewriting
system. For the rest of this section, when we say a word over the
alphabet $A$ is irreducible, we mean that it is irreducible with
respect to the rewriting system $S$.

The submonoid of right units $R$ is generated by the prefixes of the
words from $\Delta$. Indeed, let $I$ be the set of non-empty prefixes
of words from $\Delta$, that is,
\[
I = \{ x \in A^+ \mid xy \in \Delta \ \mbox{for some} \
y \in A^* \}.
\]
Clearly all words in the set $I$ represent right invertible elements
of $M$. Conversely, we have the following result.

\begin{Lemma}\emph{\cite[Lemma~3.3]{Zhang}} Let $u \in A^*$ be
  irreducible modulo $S = S(T)$. If $[u]_T$ is right invertible, then $u \in I^*$. \end{Lemma}

It follows from this lemma that $I$ constitutes a finite generating
set for the submonoid $R$ of right units of the monoid $M$ (that is, the submonoid of all right invertible elements). 
Furthermore, in \cite[Theorem~4.4]{Zhang} Zhang proves the following result which describes the structure of the 
submonoid of right units of the monoid $M$. 

\begin{Thm}\emph{
\cite[Theorem~4.4]{Zhang}
}
Let $M$ be a finitely presented special monoid. 
The submonoid of right units $R$ of $M$ is    
a free product of the group of units $G$ and a finitely generated free monoid.
\end{Thm}
Monoid free products will be formally defined in Section~\ref{sec_amalg} below.
Our next goal is to show that the Cayley graph of a special monoid
has a tree-like structure. The action of the monoid on the
corresponding tree will be used to construct a free resolution of the
trivial module.

Let
$\mathcal T$ be the set of irreducible words in $A^*$ with no suffix
in $I$.

\begin{Lemma}\label{l:preserve.irred} Let $w\in\mathcal T$ and let
  $u\in A^*$ be irreducible. Then $wu$ is irreducible. \end{Lemma}
\begin{proof} If $wu$ is not irreducible, then since both $w$ and $u$
  are irreducible it follows that $w=xy$ and $u=zw$ with
  $yz$ a left-hand side of a rewrite rule and $y,z$ both non-empty. But
  every left hand side of a rewrite rule is in $\Delta^*$ and so $y$
  has a non-empty suffix $v$ that is a prefix of an element of
  $\Delta$. But then $v\in I$, contradicting that $w\in \mathcal T$.
\end{proof}

We recall the definition of the pre-order $\leq_{\gr}$ on the monoid
$M$. For all $m,n \in M$ we write $m \leq_{\gr} n$ if and only if $mM
\subseteq nM$, and write $m \gr n$ if $m \leq_{\gr} n$ and $n
\leq_{\gr} m$. Obviously $\gr$ is an equivalence relation on $M$, usually called Green's $\mathscr{R}$-relation, and
$M / \gr$ is a poset with the order induced by $\leq_{\gr}$.
In terms of the right Cayley graph $\Gamma(M,A)$ of $M$ we have $m \leq_{\gr} n$ if and
only if there is a directed path from $n$ to $m$, while the
$\gr$-classes 
are the vertex sets of the 
Sch\"{u}tzenberger graphs of the monoid.

Let $\mathcal{L}$ be a subset of $A^*$ containing the empty word. For
any two words $\alpha, \beta \in \mathcal{L}$ write
$\alpha \preceq \beta$ if and only if $\beta$ is a prefix of $\alpha$.
This defines a poset which we denote by $P_{\mathcal{L}}$. This poset
is the reversal of the prefix order on the set of words $\mathcal{L}$.
This poset is countable since $A$ is finite. The empty word is the
unique maximal element of the poset.  This poset is
\emph{locally-finite} in the sense that every interval $[x,y]$ in this
poset contains finitely many elements. 
In fact the principal filter of every
element in this poset is finite since a word admits only finitely many
prefixes.
Recall that if $s$ and $t$ are elements of a poset $P$ then
we say $s$ \emph{covers} $t$ if $s < t$ and $[s,t] = \{s,t\}$. A locally
finite poset is completely determined by its cover relations. The
\emph{Hasse diagram} of a poset $P$ is a graph whose edges are the
cover relations. Hasse diagrams are drawn in such a way that
if $s < t$ then $t$ is drawn with a higher vertical coordinate than
$s$.

\begin{Prop}\label{prop:tree}
Let $\mathcal L\subseteq A^*$ contain the empty word.  Then the Hasse diagram of $P_{\mathcal{L}}$ is a rooted tree (with root the empty word).
\end{Prop}
\begin{proof}
For $n\geq 0$, let  $\mathcal L_n$ consist of those words from $\mathcal L$ of length at most $n$.  Let $\Lambda$ (respectively, $\Lambda_n$) be the Hasse diagram of $P_{\mathcal{L}}$ (respectively, $P_{\mathcal{L}_n}$).  Then $\Lambda=\varinjlim \Lambda_n$ and hence, since a direct limit of trees is a tree, it suffices to handle the case that $\mathcal L$ is finite.  We proceed by induction on $|\mathcal L|$.  If $|\mathcal L|=1$, then $\Lambda$ consists of a single vertex and there is nothing to prove.  Assume true for languages with at most $n$ elements and suppose that $\mathcal L$ has $n+1$ elements.  Suppose that $w\in \mathcal L$ has maximum length.  Let $v$ be the longest proper prefix of $w$ belonging to $\mathcal L$ (it could be the empty word).   Let $\Lambda'$ be the Hasse diagram of $P_{\mathcal{L}\setminus \{w\}}$; it is a rooted tree with root the empty word by induction.   Then there is an edge between $v$ to $w$ in $\Lambda$ and that is the only edge incident on $w$.  Hence $\Lambda$ and $\Lambda'$ have the same Euler characteristic and so $\Lambda$ is a tree (as $\Lambda'$ was).
\end{proof}

  It is possible for an element of $P_{\mathcal{L}}$ to cover infinitely
many distinct elements of $P_{\mathcal{L}}$. For example, if
$\mathcal{L}= \{\epsilon, ab, aab, aaab, aaaab, \ldots \}$ then
$\epsilon$ covers all the other words in this set.

The following fact is essentially established
in~\cite[Lemma~5.2]{Zhang} and the discussion afterwards.

\begin{Prop}\label{p:transversal}
  Every element $m\in M$ can uniquely be expressed in the form
  $m=[w_m]u_m$ with $w_m\in \mathcal T$ and $u_m\in R$. Moreover, the
  irreducible word $v\in A^*$ representing $m$ is $w_mt$ where
  $t\in I^*$ is the longest suffix of $v$ in $I^*$ and $[t]=u_m$.
  Furthermore, if $m,n\in M$, then $m\leq_{\mathscr R} n$ if and only
  if $w_n$ is a prefix of $w_m$. Hence the Hasse diagram of
  $M/\mathscr R$ is a tree rooted at $1$. \end{Prop}

\begin{proof}
  Let $v\in A^*$ be the irreducible word with $[v]=m$. Then $v=v'v''$
  where $v''$ is the longest suffix in $I^*$. It follows that
  $v'\in \mathcal T$ and $v''$ represents an element of $R$. This
  shows the existence of such a factorization. For uniqueness, let
  $w\in \mathcal T$ and $x\in A^*$ be an irreducible word representing
  an element of $R$. By~\cite[Lemma~3.3]{Zhang}, we have that
  $x\in I^*$. Then $wx$ is irreducible by
  Lemma~\ref{l:preserve.irred}. Thus $wx=v'v''$. By choice of $v''$,
  we must have $|x|\leq |v''|$. If $|x|<|v''|$, then some non-empty
  prefix of $v''$ is a suffix of $w$. As $I$ is prefix-closed, whence
  so is $I^*$, this contradicts that $w\in \mathcal T$. Thus $x=v''$
  and hence $w=v'$. This establishes the uniqueness of the
  decomposition.

Suppose now that $m=nn'$ with $n'\in M$. Let $z$ be a right inverse of $u_m$ and let $v$ be an irreducible word representing $u_nn'z$.  Then $w_nv$ is an irreducible word representing $nn'z=mz=[w_m]u_mz=[w_m]$ by Lemma~\ref{l:preserve.irred}.  Thus $w_m=w_nv$ and so $w_n$ is a prefix of $w_m$.  Conversely, suppose that $w_n$ is a prefix of $w_m$.  Clearly, $[w_n]\mathrel{\mathscr R} n$ and $[w_m]\mathrel{\mathscr R} m$ as $u_m,u_n$ are right invertible.  So it suffices to observe that $[w_m]\leq_{\mathscr R} [w_n]$.

The final statement follows from Proposition~\ref{prop:tree}.
\end{proof}

Retaining the notation of Proposition~\ref{p:transversal} we obtain the following immediate corollary.

\begin{Cor}\label{c:free.right}
The action of $R$ on the right of $M$ is free with transversal $\mathscr T=\{[w]\mid w\in\mathcal T\}$.  Furthermore, $M/R\cong M/\mathscr R$.
\end{Cor}

Another corollary is that all principal right ideals of $M$ are isomorphic as right $M$-sets.

\begin{Cor}\label{c:ISO.princ.right}
Let $n\in M$.  Then the mapping $\varphi_n\colon M\to nM$ given by  $\varphi_n(m)=[w_n]m$ is an isomorphism of right $M$-sets.
\end{Cor}
\begin{proof}
As $nM=[w_n]M$, the map $\varphi_n$ is clearly a surjective homomorphism of right $M$-sets.  To see that is an isomorphism, suppose that $\varphi_n(m)=\varphi_n(m')$.  Let $v,v'\in A^*$ be irreducible words representing $m,m'$, respectively.  Then $w_nv$ and $w_nv'$ are irreducible by Lemma~\ref{l:preserve.irred}.  As they represent the same element of $M$, we deduce that $v=v'$ and so $m=m'$.
\end{proof}

We now generalize Corollary~\ref{c:ISO.princ.right} to show that every right ideal of $M$ is a free $M$-set.

\begin{Thm}\label{t:right.free}
Let $M$ be a special monoid.  Then every right ideal of $M$ is a free right $M$-set and dually every left ideal of $M$ is a free left $M$-set.
\end{Thm}
\begin{proof}
Let $X$ be a right ideal of $M$ and let $X'=\{w\in \mathcal T\mid [w]\in X\}$.  Let $U'$ be the set of elements $w\in X'$ with no proper prefix in $X'$.  We claim that $X$ is freely generated as an $M$-set by $U=\{[w]\mid w\in U'\}$.  By Proposition~\ref{p:transversal} if $s,t\in U$ are distinct, then $sM\cap tM=\emptyset$.  Indeed, if $m\in sM\cap tM$, then $w_m$ has both $w_s$ and $w_t$ as prefixes and hence either $w_s$ is a prefix of $w_t$, or vice versa, contradicting the definition of $U'$.  Also, by Corollary~\ref{c:ISO.princ.right}, for each $s\in U$, we have that $sM\cong M$ as a right $M$-set.  It follows that $U$ freely generates a sub-$M$-subset $Y$ of $X$.  We show that $Y=X$.

If $m\in X$, then $m=[w_m]u_m$ with $w_m\in \mathcal T$ and $u_m\in R$.  Then $[w_m]\in X$ as $[w_m]\gr m$.  Let $w\in \mathcal T$ be the shortest prefix of $w_m$ with $[w]\in X$.  Then $w\in U'$ and  $m\in [w]M\subseteq Y$.  This completes the proof.
\end{proof}

\begin{Rmk}\label{r:fg}
Note that if $X$ is a free right $M$-set on a subset $B$ and if $X$ has a finite generating set, then $B$ is finite.  Indeed, if $C$ is a finite generating set for  $X$, then there is a finite subset $B'\subseteq B$ such that $C\subseteq B'M$.  But then $B\subseteq B'M$ and hence $B=B'$ by freeness of the action.
\end{Rmk}

Let $\Gamma(M,A)$ be the Cayley graph of $M$ with respect to $A$.  Let $\Gamma(M,A,m)$ denote the strongly connected component of $m$ (also called the Sch\"utz\-en\-ber\-ger graph of $m$).    An immediate geometric consequence of Corollary~\ref{c:ISO.princ.right} is the following.

\begin{Cor}\label{c:Cayley.geometry}
Let $n\in M$.  Then there is an isomorphism of $A$-labeled graphs  $\Gamma(M,A,1)\to \Gamma(M,A,n)$ sending $1$ to $[w_n]$.  If $\Gamma_n$ is the induced subgraph of $\Gamma(M,A)$ consisting of all vertices accessible from $n$, then $\Gamma(M,A)$ is isomorphic to $\Gamma_n$ as an $A$-labeled graph via an isomorphism taking $1$ to $[w_n]$.
\end{Cor}

Corollary~\ref{c:Cayley.geometry} recovers as a special case the result~\cite[Theorem~4.6]{Malheiro2005} that
all the maximal subgroups of a special monoid are isomorphic to each
other.  This is because the
Sch\"{u}tzenberger group of a regular $\gr$-class is isomorphic to the
automorphism group of its labelled Sch\"{u}tzenberger graph
\cite[Theorem~3]{Stephen1996}.

Next we wish to show that there is a unique edge entering any strongly connected component of $\Gamma(M,A)$ other than the strong component of $1$, and that it ends at an element of $\mathscr T$ (see Corollary~\ref{c:free.right} for the notation).  Let us say that an edge of a digraph \emph{enters} a strong component $C$
of the graph
if its initial vertex is not in $C$ and its terminal vertex is in $C$.

\begin{Prop}\label{p:entrance}
Let $n\in \mathscr T\setminus \{1\}$ (and so $n=[w_n]$).  Then if $w_n=xa$ with $a\in A$, we have that $[x]>_{\mathscr R} n$, $[a]\notin R$ and $[x]\xrightarrow{\,\,a\,\,}n$ is the unique edge entering $\Gamma(M,A,n)$.
\end{Prop}
\begin{proof}
Note that $x$ is irreducible.  Let $x=x'x''$ with $x''$ the longest suffix of $x$ in $I^*$.  Then $x'=w_{[x]}$ and $w_n$ is not a prefix of $x'$.  Thus $[x]>_{\mathscr R} [w_n]=n$ by Proposition~\ref{p:transversal}.  It follows that $[x]\xrightarrow{\,\,a\,\,}n$ enters $\Gamma(M,A,n)$ and hence $a\notin R$.

Suppose that $m\xrightarrow{\,\,b\,\,}m'$ enters $\Gamma(M,A,n)$.  Let $w$ be an irreducible word representing $m$.  Then $w=w_my$ where $y\in I^*$ is the longest suffix of $w$ in $I^*$.  We claim that $wb$ has no suffix in $I$.  Indeed, if it did, then since  $I$ is prefix-closed and $w_m$ has no suffix in $I$, we must have that $yb$ has a suffix in $I$.  Then $yb=rs$ where $s\in I$.  Since $r$ is a prefix of $y$ and $I$ (and hence $I^*$) is prefix-closed, we obtain that $yb=rs\in I^*$.  Thus $yb$ represents an element of $R$ and so \[m'=[w_myb]\mathrel{\mathscr R}[w_m]\mathrel{\mathscr R} m\] a contradiction.  Thus $wb$ has no suffix in $I$.

We claim that $wb$ is irreducible.   Suppose that $wb$ is not irreducible.  Then since $w$ is irreducible, each left-hand side in the rewriting system belongs to $\Delta^*$ and $\Delta\subseteq I$, we must have that $wb$ has a suffix in $I$, a contradiction.

Putting it all together, we deduce that $wb\in \mathcal T$ and so $wb=w_n$ by Proposition~\ref{p:transversal}.  It follows that $b=a$ and $w=x$, completing the proof.
\end{proof}

Let $\Gamma$ be the directed graph obtained from $\Gamma(M,A)$ by collapsing each strongly connected component (and its internal edges) to a point.  So the vertex set of $\Gamma$ is $M/\mathscr R$ and there is an edge $(m,a)$ from the $\mathscr{R}$-class $R_m$ of $m$ to the $\mathscr{R}$-class $R_{ma}$ of $ma$ if $m\in M$, $a\in A$ and $R_m\neq R_{ma}$.  We aim to show that $\Gamma$ is a regular rooted tree isomorphic to the Hasse diagram of $M/\mathscr{R}$.  Note that this tree can be of infinite degree.

\begin{Thm}\label{t:looks.like.Hasse}
The graph $\Gamma$ is isomorphic as a digraph to the Hasse diagram of $M/\mathscr{R}$ ordered by $\geq_{\mathscr R}$.  This graph is a regular rooted tree with root the strong component of $1$.
\end{Thm}
\begin{proof}
We retain the above notation.
Suppose first that $w,w'\in \mathcal T$ and there is an edge from $\Gamma(M,A,[w'])$ to $\Gamma(M,A,[w])$; it is unique by Proposition~\ref{p:entrance}.  Then, by Proposition~\ref{p:entrance}, we have that if $w=xa$ with $a\in A$, then $[x]\mathrel{\mathscr R} [w']$.  Thus if $x'$ is the longest suffix of $x$ belonging to $I^*$, then $x=w'x'$ and $w=w'x'a$.   Since $I$ is prefix-closed, it follows that if $y$ is any non-empty prefix of $x'$, then $w'y$ has a suffix in $I$ and hence does not belong to $\mathcal T$.  Thus in the prefix order on $\mathcal T$, there is no element between $w'$ and $w$.  It follows from Proposition~\ref{p:transversal} that in the Hasse diagram of $M/\mathscr{R}$ with respect to $\geq_{\mathscr R}$, there is an edge from $R_{[w']}$ to $R_{[w]}$.

Conversely, suppose that there is an edge in the Hasse diagram from
$R_{[w']}$ to $R_{[w]}$ with $w,w'\in \mathcal T$.  Then $w'$ is a
proper prefix of $w$ by Proposition~\ref{p:transversal} and so $w=w'y$
with $y\in A^*$ irreducible  and non-empty. Let $a\in A$ be the last
letter of $y$, so $y=y'a$.  Then $[w']\leq_{\mathscr
  R}[w'y']\leq_{\mathscr R}[w]$ and so one of these inequalities is an
equality.  Since $w$ is not a prefix of $w'y'$, it follows
from Proposition~\ref{p:transversal} (or by ~\cite[Lemma~5.2]{Zhang})
that the second inequality is strict. Thus $[w'y']$ belongs to the
strong component of $[w']$ and the image of the edge
$[w'y']\xrightarrow{\,\,a\,\,}[w]$ connects the strong component of
$[w']$ to the strong component of $[w]$ in $\Gamma$ (and is the only
such edge by Proposition~\ref{p:entrance}).

Since the reverse prefix order on any set of words containing the empty word is a rooted tree, it follows that $\Gamma$ is a rooted tree with root the strong component of $1$.  By construction of $\Gamma$ and Corollary~\ref{c:Cayley.geometry} it follows that all vertices have the same cardinality set of children.
\end{proof}

Note that in general if $M$ is a monoid generated by a finite set $A$, and if $R'$ and $R''$ are $\gr$-classes of $M$ such that
$R'$ covers $R''$ in the poset $M / \gr$, then there must exist
elements $x \in R'$ and $y \in R''$ and a generator $a \in A$ such
that $xa = y$ in $M$. The second part of the proof of the above theorem shows that in a finitely generated special monoid in this situation there are unique elements $x \in R', y \in R''$ and $a \in A$ satisfying these properties.

We note that the left action of $M$ on $\Gamma(M,A)$ induces a left action of $M$ on $\Gamma$ by cellular mappings since strong components are mapped into strong components.  However, elements of $m$ can collapse edges to a point.  In fact, $\Gamma$ (being a tree) is a simplicial graph ($1$-dimensional simplicial complex) and $M$ acts by simplicial mappings.
For example,  consider the bicyclic monoid $B=\langle a,b\mid ab=1\rangle$. 
Then since $a \mathrel{\mathscr R} 1$ left multiplication by $a$ collapses the vertices corresponding to the  strong components of $1$ and $b$ and hence collapses the edge between these components.

We can view the vertex set of $\Gamma$ as $M/R$ and so if we use the simplicial chain complex for $\Gamma$, we have $C_0(\Gamma)\cong \mathbb Z[M/R]\cong \mathbb ZM\otimes_{\mathbb ZR}\mathbb Z$ as a $\mathbb ZM$-module.    We can identify $C_1(\Gamma)$ as a $\mathbb ZM$-module with  the quotient $C_1(\Gamma(M,A))/N$ where $N$ is the $\mathbb ZM$-submodule generated as an abelian group by edges $m\xrightarrow{\,\,a\,\,}ma$ with $a\in A$ and $m\mathrel{\mathscr R} ma$.    Note that  $C_1(\Gamma(M,A))$ is a free $\mathbb ZM$-module of rank $|A|$.  We shall show that $N$ is a free $\mathbb ZM$-module of finite rank, as well.  It will then follow that $C_1(\Gamma)$ is of type $\FP_{\infty}$ with projective dimension at most $1$.

Note that $N$ is the direct sum over all $a\in A$ of the submodules $N_a$ spanned by edges $m\xrightarrow{\,\,a\,\,}ma$ with $m\mathrel{\mathscr R}ma$ and so it suffices to show that each of these submodules $N_a$  is a finitely generated free $\mathbb ZM$-module.

\begin{Prop}\label{p:loops.free}
Let $a\in A$.  Then $N_a$ is a finitely generated free $\mathbb ZM$-module.  Consequently, $N$ is a finitely generated free $\mathbb ZM$-module.
\end{Prop}
\begin{proof}
Let $L=\{m\in M\mid m\mathrel{\mathscr R} ma\}$.  Then $L$ is a left ideal of $M$ and $N_a\cong \mathbb ZL$.  First observe that if $a\in R$, then $L=M$ and there is nothing to prove.   So assume that $a\in A\setminus R$.  By Theorem~\ref{t:right.free} we have that $L$ is a free left $M$-set.  By Remark~\ref{r:fg} it suffices to prove that $L$ is finitely generated.

We claim that $L$ is generated by $I'=\{[w]\in L\mid w\in I\}$ , which is finite as $I$ is finite.  Let $m\in L$ and let $w\in A^*$ be irreducible with $[w]=m$.  There are two cases.  Assume first that $wa$ is irreducible. Then since $ma\mathrel{\mathscr R}m$, it follows from Proposition~\ref{p:transversal} that $wa\notin\mathcal T$ (as $wa$ is not a prefix of $w$) and so $wa=sxa$ with $xa\in I$.  Since $a\notin R$, we must have that $x$ is non-empty.   Since $I$ is prefix-closed, $x\in I$.  Thus $[x],[xa]\in R$ and hence $[x]\mathrel{\mathscr{R}}[x]a$.  Then $m=[s][x]$ and $[x]\in I'$.  So $m\in MI'$.

Next assume that $wa$ is not irreducible.   Then $wa=sxa$ with $xa\in \Delta$, as $w$ is irreducible.  But $a\notin R$ and so $x$ is non-empty.  Thus $x\in I$.  Also $xa\in \Delta\subseteq I$.  Thus $[x],[xa]\in R$ and so $[x]\mathrel{\mathscr R}[x]a$.  Also, $m=[s][x]$ with $[x]\in I'$ and so $m\in MI'$.  This completes the proof.
\end{proof}

Now all is in place to prove the first main result of this section.

\begin{Thm}\label{t:special}
Let $M$ be a finitely presented special monoid with group of units $G$.
\begin{enumerate}
\item If $G$ is of type $\FP_n$ with $1\leq n\leq \infty$, then $M$ is of type left-$\FP_n$ and of type right-$\FP_n$.
\item $\mathop{\mathrm{cd}} G\leq \mathrm{left \; cd} M \leq \max\{2,\mathop{\mathrm{cd}} G\}$ and
$\mathop{\mathrm{cd}} G\leq \mathrm{right \; cd} M\leq \max\{2,\mathop{\mathrm{cd}} G\}$.
\end{enumerate}
\end{Thm}
\begin{proof}
We retain the above notation.
We prove the results for left-$\FP_n$ and left cohomological dimension (the other results are dual).
First note that if $L$ denotes the submonoid of left invertible elements, then $M$ is a free left $L$-set by the dual of Proposition~\ref{p:transversal}.  If $B$ is the basis of $M$ as a left $L$-set, then each element $m\in M$ can be expressed uniquely as $u_mb_m$ with $b_m\in B$ and $u_m\in L$.  But then if $g\in G$ with $gm=m$, we must have $gu_mb_m=u_mb_m$.  It follows that $gu_m=u_m$ by uniqueness.  But since $L$ is a free product of $G$ with a finitely generated free monoid by~\cite[Theorem~4.4]{Zhang}, it follows that $G$ acts freely on the left of $L$ and so $g=1$.    Thus $\mathbb ZM$ is a free left $\mathbb ZG$-module and so $\mathop{\mathrm {cd}} G\leq \mathop{\mathrm {left \; cd}} M$ as any projective resolution of $\mathbb Z$ over $\mathbb ZM$ is a projective resolution over $\mathbb ZG$.

The graph $\Gamma$ is a tree with a simplicial action by $M$ described above.  So we have an exact sequence of $\mathbb ZM$-modules
\[0\longrightarrow C_1(\Gamma)\longrightarrow C_0(\Gamma)\longrightarrow \mathbb Z\longrightarrow 0.\] We have identified $C_1(\Gamma)\cong C_1(\Gamma(M,A))/N$ where $C_1(\Gamma(M,A))$ is free of rank $|A|$ and $N$ is a finitely generated free module by Proposition~\ref{p:loops.free}.  Thus $C_1(\Gamma)$ is of type $\FP_{\infty}$ and has projective dimension at most $1$.

On the other hand, $C_0(\Gamma)\cong \mathbb Z[M/R]\cong \mathbb ZM\otimes_{\mathbb ZR} \mathbb Z$.  By Zhang's theorem~\cite[Theorem~4.4]{Zhang}, $R=G\ast C^*$ where $C$ is a finite alphabet, and hence $R$ is of type $\FP_n$ whenever $G$ is, and, $\mathop{\mathrm{left \; cd}} R\leq \max\{1,\mathop{\mathrm{cd}} G\}$ by~\cite[Theorem~5.5]{CremannsOtto1998} (or see Corollaries~\ref{c:free.prod.fp} and~\ref{c:free.prod.geom} below).  Note that a finitely generated free monoid is of type $\FP_{\infty}$ and of cohomological dimension $1$ because its Cayley graph is a tree and a free $M$-CW complex of finite type of dimension $1$.

As $\mathbb ZM$ is a free, and hence flat, right $\mathbb ZR$-module
by Corollary~\ref{c:free.right}, it follows from Lemma~\ref{l:flat.base} that $C_0(\Gamma)\cong \mathbb ZM\otimes_{\mathbb ZR}\mathbb Z$ is of type $\FP_n$ and of projective dimension at most $\max\{1,\mathop{\mathrm{cd}} G\}$.

The result now follows from an application of Corollary~\ref{c:fp.resolved}.
\end{proof}

In general, the left- and right-cohomological dimensions of a monoid are not equal. In fact they are completely independent of each other; see~\cite{Guba1998}. One immediate corollary of the above result is that if $M$ is a finitely presented special monoid with left- and right-cohomological dimensions both at least equal to $2$, then the left cohomological dimension of $M$ is equal to its right cohomological dimension.

As an application of Theorem~\ref{t:special} we now show how it can be used to prove that all special one-relator monoids are of type $\FPinfty$, answering a case of a question of Kobayashi. We also recover Kobayashi's result
(see~\cite[Theorem~7.2]{Kobayashi1998} and~\cite[Corollary 7.5]{Kobayashi2000})
that if the relator is not a proper power then the cohomological dimension is at most $2$.

A word $u \in A^*$ is called \emph{primitive} if it is not a proper power in $A^*$.

\begin{lem}\emph{\cite[Corollary~4.2]{LyndonSchutz1962}}
For every nonempty word $w \in A^*$ there is a unique primitive word $p$ and a unique integer $k \geq 1$ such that $w = p^k$.
\end{lem}

The following lemma is well known. We include it here for completeness.

\begin{lem}\label{lem_proppower}
Let $M=\langle A\mid w=1\rangle$. Write $w = p^k$ where $p$ is a primitive word and $k \geq 1$. The group of units $G$ of $M$ is a one-relator group with torsion if and only if $k > 1$.
\end{lem}
\begin{proof}
Since it is a prefix and suffix of $w$, it follows that $p$ is invertible in $M$. Therefore, the decomposition of $w$ into indecomposable invertible factors has the form $w = (p_1 p_2 \ldots p_l)^k$ where $p_1 p_2 \ldots p_l$ is the decomposition of $p$ into indecomposable invertible factors.
Let $P = \{ p_i : 1 \leq i \leq l \} \subseteq A^*$. Let $X = \{ x_{p} : p \in P \}$ be an alphabet in bijection with the set of words $P$, so distinct words $p_i$ and $p_j$ from $P$ correspond to distinct letters $x_{p_i}$ and $x_{p_j}$ from the alphabet $X$. It follows from~\cite[Lemma 96]{Adjan1966} that the group of units of the monoid $M$ is isomorphic to the group defined by the group presentation $\mathrm{Gp}\langle X \mid (x_{p_1} x_{p_2} \ldots x_{p_l})^k=1 \rangle$. Observe that $x_{p_1} x_{p_2} \ldots x_{p_l} \in X^*$, i.e. this is a positive word over the alphabet $X$. In particular the word $(x_{p_1} x_{p_2} \ldots x_{p_l})^k$ is cyclically reduced. Since the word $p_1 p_2 \ldots p_l$ is primitive by assumption it follows that the word $x_{p_1} x_{p_2} \ldots x_{p_l} \in X^*$ is also primitive. Hence $(x_{p_1} x_{p_2} \ldots x_{p_l})^k$ is a proper power if and only if $k > 1$. But then by a well-known result of Karrass, Magnus and Solitar characterising elements of finite order in one-relator groups~\cite[Theorem~5.2]{LyndonAndSchupp} it follows that the group of units of $M$ is a one-relator group with torsion if and only if $k > 1$.
\end{proof}

Well-written accounts of the result~\cite[Lemma 96]{Adjan1966} of Adjan used in the previous proof may be found in~\cite[Section~1]{Lallement1974} and~\cite[Section~2]{Lallement1988}.
The following result gives a positive answer to Kobayashi's question \cite[Problem~1]{Kobayashi2000} in the case of special one-relator monoids 

\begin{Cor}\label{c:special.one} 
Let $M$ be the one-relator monoid  $\langle A\mid w=1\rangle$. 
 Then $M$ is of type left- and right-$\FP_{\infty}$. 
Moreover, if $w$ is not a proper power then
  $\mathop{\mathrm{left \; cd}} M\leq 2$ and $\mathop{\mathrm{right \; cd}} M\leq 2$, and otherwise
  $\mathop{\mathrm{left \; cd}} M=
\mathop{\mathrm{right \; cd}} M=
\infty$. \end{Cor}

\begin{proof}
We prove the results for left-$\FP_\infty$ and left cohomological dimension (the other results are dual).
The group
  of units $G$ of $M$ is a one-relator group by Adjan's theorem~\cite[Lemma 96]{Adjan1966} (this
  also follows from the results of Zhang described above), and
  hence of type $\FP_{\infty}$ by Lyndon's theorem~\cite{Lyndon1950}. This proves the
  first statement in light of Theorem~\ref{t:special}. The second
  statement follows since by Lemma~\ref{lem_proppower} the group
$G$ is a
  one-relator group whose defining relator is not a proper power in
  the first case and is a proper power in the second. By a theorem of
  Lyndon~\cite{Lyndon1950} $G$ has cohomological dimension at most $2$ in the first case
  and has infinite cohomological dimension in the second. The result
  now follows from Theorem~\ref{t:special}. \end{proof}

We now turn our attention to proving the topological analogue of Theorem~\ref{t:special}. We do this by showing how an equivariant classifying space for a special monoid may be constructed from an equivariant classifying space for its group of units.

Note that while for finitely presented monoids it follows from~\cite{GraySteinberg1} that the properties left $\FP_n$ and left $\F_n$ are equivalent, in contrast it is not known whether $\mathrm{left \; cd}(M)$ and $\mathrm{left \; gd}(M)$ coincide (this is even open for groups). Therefore, the second part of the following theorem is not an immediate consequence of Theorem~\ref{t:special}.

\begin{Thm}\label{t:special:topological}
Let $M$ be a finitely presented special monoid with group of units $G$.
\begin{enumerate}
\item If $G$ is of type $\F_n$ with $1\leq n\leq \infty$, then $M$ is of type left- and right-$\F_n$.
\item $\mathop{\mathrm{gd}} G\leq \mathop{\mathrm{left \; gd}} M \leq \max\{2,\mathop{\mathrm{gd}} G\}$
and
$\mathop{\mathrm{gd}} G\leq \mathop{\mathrm{right \; gd}} M \leq \max\{2,\mathop{\mathrm{gd}} G\}$.
\end{enumerate}
\end{Thm}
\begin{proof}
We prove the results for left-$\F_n$ and left geometric dimension. The other results are dual.
It is proved in~\cite[Section~6]{GraySteinberg1} for finitely presented monoids the properties left-$\F_n$ and left-$\FP_n$ coincide. Now part (1) of the theorem follows from the first part of Theorem~\ref{t:special}. (One can also see this directly from the construction below.)

To prove part (2), first note that we showed that $M$ was a free left $G$-set at the beginning of the proof of Theorem~\ref{t:special}.  Hence any free $M$-CW complex is a free $G$-CW complex.  Also note that Theorem~\ref{t:right.free} implies that every projective $M$-set is free, as $Me$ is a left ideal for any idempotent $e$.  Thus any projective $M$-CW complex $X$ is a free $M$-CW complex and so it follows that $G\backslash X$ is $K(G,1)$-space. The inequality $\mathop{\mathrm{gd}} G\leq \mathop{\mathrm{right \; gd}} M$ follows.

We  shall now explain how to construct an equivariant classifying space for $M$ of dimension $\max\{2, \mathrm{gd}(G) \}$.

Let $X_G$ be an equivariant classifying space for the group $G$.
Since $G$ is a group it follows that the projective $G$-CW complex $X_G$ is
a free $G$-CW complex.
By Zhang's theorem~\cite[Theorem~4.4]{Zhang}, the submonoid of right units $R$ of $M$ is isomorphic to the monoid free product $G\ast C^*$ where $C^*$ is a free monoid over a finite alphabet $C$. The right Cayley graph $\Gamma(C^*)$ of $C^*$ with respect to the generating set $C$ is a tree and thus is a free equivariant classifying space for the monoid $C^*$. In particular $C^*$ is of geometric dimension at most $1$. Let $X$ be the left equivariant classifying space for $R \cong G \ast C^*$ given by the construction in the proof of Theorem~\ref{t:bass.serre.free} in Section~\ref{sec_amalg} below.
From the construction it follows that $X$ is a free $R$-CW complex and an equivariant classifying space for $R$.  (If $X_G$ has a $G$-finite $n$-skeleton, then $X$ has an $R$-finite $n$-skeleton.)
It also follows from the construction of $X$ that
$\dim X \leq \max\{1, \dim X_G\}$ (compare with Theorem~\ref{t:amalg.cd}).

Now $M$ is an $M$-$R$-biset, which is free as a left $M$-set and is also free as a right $R$-set by Corollary~\ref{c:free.right}, and $X$ is a free left $R$-CW complex. It follows
from Proposition~\ref{c:base.change.cw}
that $M \otimes_R X$ is a free left $M$-CW complex with $\dim M \otimes_R X=\dim X$. (It will have $M$-finite $n$-skeleton if $X$ has $R$-finite $n$-skeleton.) The complex $M \otimes_R X$ is a disjoint union of copies of $X$, one for each $\mathscr{R}$-class of $M$ by Remark~\ref{r:tensor.with.free}.  To make this concrete, take the transversal $\mathcal{T}$ of the $\mathscr{R}$-classes of $M$ defined above, which is a basis for $M$ as a free right $R$-set.  Then each element of $M\otimes_R X$ can be uniquely written in the form $t\otimes x$ with $t\in \mathcal T$ and $x\in X$ and $M\otimes_R X=\coprod_{t\in \mathcal T} t\otimes X$.   We say that two elements $m \otimes x$ and $m' \otimes x'$ of $M \otimes_R X$ belong to the same copy of $X$ in $M \otimes_R X$ if and only if $m \mathrel{\mathscr{R}} m'$.

Fix a basepoint $x_0 \in \mathcal{Q} \subseteq X_0$. Next we connect the space $M \otimes_R X$ by attaching edges $m \otimes x_0 \rightarrow ma \otimes x_0$ for each $m \in M$ and $a \in A$. This is the same as attaching a free $M$-cell $M \times B^1$ of dimension $1$ based at $1 \otimes x_0 \rightarrow a \otimes x_0$ for each $a \in A$. Let $Y$ denote the
resulting free $M$-CW complex. The $\mathscr{R}$-order in the monoid $M$ induces in a natural way an order  on the copies of $X$ in $Y$, and there is an edge joining two distinct copies of $X$ in $Y$ if and only if there is an edge in the right Cayley graph of $M$ joining the corresponding $\mathscr{R}$-classes. Moreover, it follows from the definition of $Y$, and Proposition~\ref{p:entrance}, that there is at most one edge joining any pair of distinct copies of $X$ in $Y$. It follows that if we contract each of the copies of $X$ in $Y$ we obtain the graph $\Gamma$ in Theorem~\ref{t:looks.like.Hasse}, which is a regular rooted tree, together with possibly infinitely many loops at each vertex. These loops arise from the edges $m \otimes x_0 \rightarrow ma \otimes x_0$ where $m \mathrel{\mathscr{R}} ma$ added in the construction of $Y$.  (Notice that if $M\otimes_R X$ has $M$-finite $n$-skeleton, then so does $Y$.)

To turn $Y$ into an equivariant classifying space for $M$ we add $2$-cells to deal with these loops, in the following way. It follows from Proposition~\ref{p:loops.free} that for each $a \in A$, the set $L=\{m\in M\mid m\mathrel{\mathscr R} ma\}$
is a free left
$M$-set generated by a finite set $F_a \subseteq L$ with $F_a \subseteq R$. For each $r \in F_a$, choose a path in $p_r$ in $1 \otimes X$ from $1 \otimes x_0$ to $1 \otimes rx_0$, choose a path $q_r$ in $1 \otimes X$ from $1 \otimes x_0$ to $1 \otimes ra x_0$, and let $e_r$ denote the edge in $Y$ labelled by $a$ from $1 \otimes rx_0$ to $1 \otimes ra x_0$. Note that since $r \in F_a \subseteq L$ it follows that $r \in R$ and $ra \in R$ and so $1 \otimes rx_0 = r \otimes x_0$ and $1 \otimes ra x_0 = ra \otimes x_0$ and hence $e_r$ is indeed one of the edges that was added during the construction of $Y$. Now for each $a \in A$ attach a free $2$-cell $M \times B^2$ to $Y$ by attaching a $2$-cell at $1 \otimes x_0$ with boundary path $p_r e_r q_r^{-1}$ and all of its translates under the action of $M$. We do this for each $a \in A$ and call the resulting complex $Z$. Now if we contract the copies of $X$ in $Z$, we obtain the tree $\Gamma$, together with loops at each vertex each of which bounds a single disk. Thus $Z$ is homotopy equivalent to the tree $\Gamma$, and hence is contractible. This shows that $Z$ is an equivariant classifying space for the monoid $M$. (Note that if $Y$ has $M$-finite $n$-skeleton, then so does $Z$ hence giving an alternative proof that if $G$ is of type $\F_n$, then $M$ is of type left-$\F_n$.)

To complete the proof, since the free $M$-CW complex $Z$ was constructed from $M \otimes_R X$ by attaching $1$-cells and $2$-cells, and since we have already observed that $\dim M \otimes_R X = \dim X \leq \max \{ 1, \dim X_G \}$, it follows that
$
\dim Z \leq \max \{ 2, \dim X_G \}
$
and hence $\mathrm{left \; gd}(M) \leq \max\{2, \mathrm{gd}(G) \}$.
\end{proof}

For special one-relator monoids we obtain the following corollary which is the topological analogue of Corollary~\ref{c:special.one}.

\begin{Cor}\label{c:special.one:topological} 
Let $M$ be the one-relator monoid  $\langle A\mid w=1\rangle$. 
Then $M$ is of type left- and right-$\F_{\infty}$. 
Moreover, if $w$  is not a proper power then
  $\mathop{\mathrm{left \; gd}} M\leq 2$ and $\mathop{\mathrm{right \; gd}} M\leq 2$, and otherwise
  $\mathop{\mathrm{left \; gd}} M=
\mathop{\mathrm{right \; gd}} M =
\infty$. \end{Cor}

In particular this results says that for every special one-relator monoid whose defining relator is not a proper power admits an equivariant classifying space of dimension at most $2$. 
In fact, in this case it turns out that the Cayley complex of the monoid gives an equivariant classifying space of dimension at most $2$,  as the following result demonstrates.

\begin{Thm}\label{thm_CayleyComplex}
Let $M=\langle A\mid w=1\rangle$ such that $w$ is not a proper power. 
Let $X$ be the $2$-complex obtained by filling in each loop labeled by $w$ in the Cayley graph $\Gamma(M,A)$ of $M$. Then $X$ is left equivariant classifying space for $M$ with dimension at most $2$.  
 \end{Thm}
 \begin{proof}
It follows from the proof of \cite[Theorem~6.14]{GraySteinberg1} that $X$ is an $M$-finite simply connected free $M$-CW complex of dimension at most 2. It is shown in \cite[Corollary~7.5]{Kobayashi2000} that the presentation $\langle A\mid w=1\rangle$ is strictly aspherical in the sense defined in \cite[Section~2]{Kobayashi1998}. The cellular chain complex of $X$ gives a free resolution displayed in Equation~(7.2) in \cite[Theorem~7.2]{Kobayashi1998}. This shows that $X$ is acyclic. Since $X$ is acyclic and simply connected it follows from the Whitehead and Hurewicz theorems that $X$ is contractible, and hence $X$ is a left equivariant classifying space for the monoid $M$. 
 \end{proof}
The analogous result to Theorem~\ref{thm_CayleyComplex} is also known to hold for one-relator groups. 
This was first observed in~\cite{Cockcroft1954} and is a consequence of Lyndon's Identity Theorem~\cite{Lyndon1950}. A more topological proof is given in~\cite{DyerVasquez1973}.

We currently do not know whether the two-sided analogues of the results proved in this section, for bi-$\Fn$ and (two-sided) geometric dimension, hold. One way to establish these results might be to seek a better understanding of the two-sided Cayley graphs of special monoids. 

As mentioned in the introduction, building on the ideas presented in this section, in \cite{GraySteinberg3}
we have extended these results to arbitrary one-relator monoids. In particular in \cite {GraySteinberg3} we give a positive answer to Kobayashi's question \cite[Problem~1]{Kobayashi2000}
by showing that every one-relator monoid $\lb A \mid u=v \rb$ is of type left- and right-$\F_\infty$ and $
\FP_{\infty}$.

\section{Amalgamated free products}
\label{sec_amalg}

For graph of groups, including free products with amalgamation and HNN extensions, there are well-established methods for constructing a $K(G,1)$ from $K(G,1)$s of the vertex and edge groups; see for example~\cite[page~92]{Hatcher2002}. This can then be used to prove results for groups about the behaviour of the properties $\Fn$ and geometric dimension for amalgamated free products and HNN extensions.  In this section, and the two sections that follow it, we use topological methods to 
investigate the behaviour of topological and homological finiteness properties of monoids, for free products with amalgamation, and HNN extension constructions.

A \emph{monoid amalgam} is a triple $[M_1,M_2;W]$ where $M_1,M_2$ are monoids with a common submonoid $W$.  The \emph{amalgamated free product} is then the pushout in the diagram
\begin{equation}\label{eq:pushout.mon}
\begin{tikzcd} W\ar{r}\ar{d} & M_1\ar{d}\\ M_2\ar{r} & M_1\ast_W M_2\end{tikzcd}
\end{equation}
in the category of monoids. Monoid amalgamated products are
\textbf{much} more complicated than group ones.
For instance, the amalgamated free product of finite monoids can have an undecidable word problem, and the factors do not have to embed or intersect in the base monoid; see~\cite{Sapir2000}.
So there are no normal forms available in complete generality that allow one construct a Bass--Serre tree.  We use instead the homological ideas of Dicks.  For more details about these methods
we refer the reader to~\cite[Chapter 1, Sections 4-7]{DicksAndDunwoody}.

An $M$-graph $X$ is a one-dimensional CW complex with a cellular action by $M$ sending edges to edges. Given an $M$-graph $X$ we use $V$ to denote its set of $0$-cells and $E$ to denote its set of $1$-cells. Given any $M$-graph, if we choose some orientation for the edges, then the attaching maps of the $1$-cells define functions $\iota, \tau$ from $E$ to $V$ where in $X$ each oriented edge $e$ starts at $\iota e$ and ends at $\tau e$.
We call $V$ and $E$ the vertex set, and edge set respectively, of the $M$-graph $X$. We shall assume that the monoid action preserves the orientation. It shall sometimes be useful to think of an $M$-graph as given by a tuple $(X,V,E,\iota,\tau)$ where $X$ is an $M$-set, $X = V \cup E$ a disjoin union where each of $V$ and $E$ is closed under the action of $M$, and $\iota, \tau\colon  E \rightarrow V$ are $M$-equivariant maps.

Let $M$ be a monoid and let $X$ be an $M$-graph.
Let $\ZV$ and $\ZE$ denote the free abelian groups on $V$ and $E$, respectively.
The cellular boundary map of $X$ is the $M$-linear map $\partial\colon \ZE \rightarrow \ZV$ with $\partial(e) = \tau e - \iota e$ for all $e \in E$. The sequence
\[\ZE
\xrightarrow{\,\,\partial\,\,}
\ZV
\xrightarrow{\,\,\epsilon\,\,}
\Z \longrightarrow 0
\]
is the augmented cellular chain complex of $X$, where $\epsilon$ is the augmentation map sending $\sum_{v \in V} n_v v$ to $\sum_{v \in V} n_v$ (i.e., each element of the basis $V$ is mapped to $1$). Throughout this section we shall frequently be confronted with the task of showing that a given $M$-graph is a tree or a forest. To do this, it is useful to recall that the
$M$-graph $X$ is
a forest if and only if $\partial\colon \ZE \rightarrow \ZV$ is injective; see~\cite[Lemmas 6.4]{DicksAndDunwoody}, i.e.,
\[0\longrightarrow \ZE
\xrightarrow{\,\,\partial\,\,}
\ZV
\xrightarrow{\,\,\epsilon\,\,}
\Z \longrightarrow 0
\]
is exact.

The results in this section improve, and give simpler proofs of, several results of Cremanns and Otto~\cite{CremannsOtto1998}
on the behaviour of $\FP_n$ under free products and certain rather restricted free products of monoids with amalgamation. The proofs in Cremanns and Otto are quite long and technical, as is often the case for results in this area. The results in this section demonstrate the type of result our topological methods were introduced to prove. They show that the topological approach may be used to prove more general results in a less technical and more conceptual way. Our results also generalise and simplify proofs of some results of Kobayashi~\cite{Kobayashi2010} on preservation of left-, right- and bi-$\FP_n$ under free products (see for example~\cite[Proposition 4.1]{Kobayashi2010}). There are no bi-$\FPn$ analogues in the literature of the two-sided results we obtain below on the behaviour of bi-$\Fn$ and geometric dimension for free products with amalgamation. Also, as far as we are aware, the results that we obtain here are the first to appear in the literature on cohomological dimension of amalgamated free products of monoids.

A monoid presentation is said to have finite homological type, abbreviated to FHT, if the, so-called, homotopy bimodule of the given presentation is finitely generated. The homotopy bimodule is a $\mathbb{Z}M$-bimodule constructed from a complex of $\mathbb{Z}A^*$-bimodules defined using the set of defining relations $R$ of the presentation $\lb A \mid \rel \rb$ of the monoid $M$, and a particular family of disjoint circuits in the derivation graph associated with the presentation. The property FHT was originally introduced by Wang and Pride~\cite{WangPride2000}.
We refer the reader to that paper, or to~\cite[Section~3]{KobayashiOtto2001}, for full details of the definition of FHT. It was proved in~\cite{KobayashiOtto2003} that for finitely presented monoids FHT and bi-$\FP_3$ (equivalently bi-$\F_3$) are equivalent. So some of the results below also have an interpretation in terms of FHT.

\subsection{The one-sided setting}

Let us define a tree $T$ for a pushout diagram \eqref{eq:pushout.mon}.  Let us assume that $f_i\colon W\to M_i$, for $i=1,2$, is the homomorphism in the diagram and put $L=M_1\ast_W M_2$ for the pushout. The right multiplicative actions of $M_1$, $M_2$ and $W$ give three different partitions of $L$ into weak orbits. Since $W \leq M_i$ the $W$-orbits give a finer partition than both the $M_1$- and $M_2$-orbits. We can then define a directed bipartite graph $T$ with one part given by the $M_1$-orbits and the other part given by the $M_2$-orbits. When an $M_1$-orbit intersects an $M_2$-orbit, that intersection will be a union of $W$-orbits,
and in this case we draw directed edges from the $M_1$-orbit to the $M_2$-orbit labelled by the $W$-orbits in this intersection.

In more detail, let $T$ be the $L$-graph with vertex set
\[V=L/M_1\coprod L/M_2\]  and edge set
\[E=L/W\] where $M_1,M_2,W$ act on the right of $L$ by first applying the canonical map to the pushout and then right multiplying.  We write $[x]_K$ for the class of $x\in L$ in $L/K$.  The edge $[x]_W$ connects $[x]_{M_1}$ with $[x]_{M_2}$ (and we usually think of it as oriented in this direction).  The incidence here is easily seen to be well defined and the action of $L$ on the left of these sets is by cellular mappings sending edges to edges and preserving orientation. Hence $T$ is an $L$-graph.

\begin{Lemma}\label{l:connected.bs.tree}
The graph $T$ is connected.
\end{Lemma}
\begin{proof}
The pushout $L$, being a quotient of the free product $M_1\ast M_2$, is generated by the images of $M_1$ and $M_2$ under the natural maps (which we omit from the notation even though they need not be injective).
We define the length of $x\in L$ to be the minimum $k$ such that $x=x_1\cdots x_k$ with $x_i\in M_1\cup M_2$.  We prove by induction on the length of $x$ that there is a path in $T$ from $[1]_{M_1}$ to $[x]_{M_1}$.  If $x=1$, this is trivial, so assume the statement is true for length $k$ and $x=x_1\cdots x_{k+1}$.  Let $p$ be a path from $[1]_{M_1}$ to $[x_2\cdots x_{k+1}]_{M_1}$.  Then $x_1p$ is a path from $[x_1]_{M_1}$ to $[x]_{M_1}$.  If $x_1\in M_1$, then $[x_1]_{M_1}=[1]_{M_1}$ and so $x_1p$ is a path from $[1]_{M_1}$ to $[x]_{M_1}$.  If $x_1\in M_2$, then $[x_1]_W$ is an edge connecting $[x_1]_{M_1}$ and $[x_1]_{M_2}=[1]_{M_2}$ and $[1]_W$ is an edge connecting $[1]_{M_1}$ with $[1]_{M_2}$ and so there is a path from $[1]_{M_1}$ to $[x]_{M_1}$.  Finally, if $x\in L$, then $[x]_{M_2}$ is connected by $[x]_W$ to $[x]_{M_1}$, which in turn is connected by a path to $[1]_{M_1}$.  Thus $T$ is connected.
\end{proof}

We aim to prove that $T$ is a tree by showing that the cellular boundary map $\partial\colon  \ZE \rightarrow \ZV$ is injective. To prove this we shall make use of semidirect products of monoids and the concept of a derivation. An account of this theory for groups may be found in~\cite{DicksAndDunwoody} where it is applied to show that the standard graph of the fundamental group of a graph of groups is a tree; see~\cite[Theorem~7.6]{DicksAndDunwoody}.

Let $M$ be a monoid and let $A$ be a left $\ZM$-module. Then we can form the semidirect product $A \rtimes M$, of the abelian group $A$ and the monoid $M$, with elements $A \times M$ and multiplication given by
\[
(a, m) (a', m') =
(a + m a', mm').
\]
The natural projection $\pi\colon  A \rtimes M \rightarrow M$, $(a,m) \mapsto m$ is clearly a monoid homomorphism. A \emph{splitting} of this projection is a monoid homomorphism $\sigma\colon  M \rightarrow A \rtimes M$ such that $\pi(\sigma(m)) = m$ for all $m \in M$. Associated to any splitting $\sigma$ of $\pi$ is a mapping $d\colon  M \rightarrow A$ defined as the unique function satisfying
\[
\sigma(m) = (d(m),m)
\]
for all $m \in M$. It follows from the fact that $\sigma$ is a homomorphism that the function $d\colon M \rightarrow A$ must satisfy
\begin{equation}\label{eq_der}
d(mm') = d(m) + md(m')
\end{equation}
for all $m, m' \in M$. Any function $d\colon M \rightarrow A$ satisfying \eqref{eq_der} is called a \emph{derivation}.
 A derivation is called \emph{inner} if it is of
the form $d(m)=ma-a$ for some $a\in A$.  It is easy to check that a
mapping $d\colon M\to A$ is a derivation if and only if
$m\mapsto (d(m),m)$ provides a splitting of the semidirect product
projection $A\rtimes M\to M$.

\begin{Lemma}\label{l:acyclic.bs.tree}
The graph $T$ is a tree.
\end{Lemma}
\begin{proof}
Since $T$ is connected by Lemma~\ref{l:connected.bs.tree}, it suffices to show that the cellular boundary map $\partial\colon \mathbb ZE\to \mathbb ZV$ is injective.  To show this, we define a  left inverse  $\beta\colon \mathbb ZV\to \mathbb ZE$.  In what follows, we abuse notation by identifying an element of $M_1$, $M_2$ or $W$ with its image in $L$.

First define $\p_1\colon M_1\to \mathbb ZE\rtimes L$ by $\p_1(m_1) = (0,m_1)$.  Then $\p_1$ is clearly a monoid  homomorphism.  Define $\p_2\colon M_2\to \mathbb ZE\rtimes L$ by $\p_2(m_2)  = ([1]_W-[m_2]_W,m_2)$.  Notice that $m_2\mapsto [1]_W-[m_2]_W$ is the inner derivation of the $\mathbb ZM_2$-module $\mathbb ZE$ associated to $-[1]_W\in \mathbb ZE$ and hence $\p_2$ is a homomorphism. Next, we observe that $\p_1f_1=\p_2f_2$. Indeed, if $w\in W$, then $\p_1f_1(w)=(0,w)$ and $\p_2f_2(w) = ([1]_W-[w]_W,w)=(0,w)$ as $[1]_W=[w]_W$.  Thus there is a well defined homomorphism $\p\colon L\to \mathbb ZE\rtimes L$ extending $\p_1,\p_2$ by the universal property of a pushout.  This map must split the semidirect product projection by construction of $\p_1,\p_2$.
Indeed, for all $m_1 \in L$ in the image of $M_1$ we have
$
\p(m_1) = \p_1(m_1) = (0,m_1)
$
and
for all $m_2 \in L$ in the image of $M_2$ we have
\[
\p(m_2) = \p_2(m_2) =
([1]_W - [m_2]_w, m_2).
\]
It follows that for all $m_1 \in L$ in the image of $M_1$ we have $\pi(\p(m_1))=m_1$, and for all $m_2 \in L$ in the image of $M_2$ we have $\pi(\p(m_2))=m_2$. Since, as already observed above, $L$ is generated by the images of $M_1$ and $M_2$ under the natural maps, and since $\pi$ and $\p$ are homomorphisms, we conclude that $\pi(\p(l))=l$ for all $l \in L$, as required.
It follows that  $\p(x) = (d(x),x)$ for some derivation $d\colon L\to \mathbb ZE$ with the property that $d(m_1)=0$ for $m_1\in M_1$ and $d(m_2) = [1]_W-[m_2]_W$ for $m_2\in M_2$.

Define $\beta\colon \mathbb ZV\to \mathbb ZE$ by  $\beta([x]_{M_1}) =d(x)$ and $\beta([x]_{M_2}) = d(x)+[x]_W$ for $x\in L$.  We must show that this is well defined.  First suppose that $x\in L$ and $m_1\in M_1$.  Then $d(xm_1)= xd(m_1)+d(x)=d(x)$ because $d$ vanishes on the image of $M_1$.   If $x\in L$ and $m_2\in M_2$, then
\begin{align*}
d(xm_2)+[xm_2]_W&=xd(m_2)+d(x)+[xm_2]_W \\&= x([1]_W-[m_2]_W)+d(x)+[xm_2]_W = d(x)+[x]_W.
\end{align*}
It follows that $\beta$ is well defined.

We now compute \[\beta\partial([x]_W) = \beta([x]_{M_2})-\beta([x]_{M_1}) = d(x)+[x]_W-d(x)=[x]_W\] for $x\in L$.  Thus $\beta\partial=1_{\mathbb ZE}$ and so $\partial$ is injective.  This completes the proof that $T$ is a tree.
\end{proof}

Since $T$ is a tree we obtain an exact sequence of $\ZL$-modules
\[
0 \longrightarrow \mathbb{Z}E
\xrightarrow{\,\,\partial\,\,} \mathbb{Z}V
\xrightarrow{\,\,\epsilon\,\,} \Z \longrightarrow 0
\]
where $E,V$ are the edge and vertex sets of $T$, respectively.
See~\cite[Theorem 6.6]{DicksAndDunwoody}.
The exactness of this cellular chain complex of $T$ can be reformulated in the following manner.

\begin{Cor}\label{c:exact.mv.begin}
There is an exact sequence of $\ZL$-modules
\[0\longrightarrow \mathbb ZL\otimes_{\mathbb ZW} \mathbb Z\longrightarrow (\mathbb ZL\otimes_{\mathbb ZM_1}\mathbb Z)\oplus (\mathbb ZL\otimes_{\mathbb ZM_2} \mathbb Z)\longrightarrow \mathbb Z\longrightarrow 0\]
 where  $L=M_1\ast_{W} M_2$ is the pushout.
\end{Cor}
\begin{proof}
This follows from the definition of $T$, the fact that  $T$ is a tree, and the observation that $\mathbb Z[L/K]\cong \mathbb ZL\otimes_{\mathbb ZK}\mathbb Z$ for $K=M_1,M_2,W$.
\end{proof}

We call $T$ the \emph{Bass--Serre tree} of the pushout.

If $f\colon X\to Y$ and $g\colon X\to Z$ are continuous mappings of topological spaces, the \emph{homotopy pushout} of $f,g$ is the space obtained by attaching $X\times I$ to $Y\coprod Z$ by the mapping $h\colon X\times \partial I\to Y\coprod Z$  with $h(x,0)=f(x)$ and $h(x,1)=g(x)$.  If $X,Y,Z$ are CW complexes and $f,g$ are cellular mappings, then $h$ is cellular and so the homotopy pushout $U$ of $f$ and $g$ is a CW complex.  If, in addition,  $X,Y,Z$ are projective $M$-CW complexes and  $f,g$ are cellular and $M$-equivariant, then $U$ is a projective $M$-CW complex by~\cite[Lemma~2.1]{GraySteinberg1}.
Moreover, by the description of the cells coming from the proof of
\cite[Lemma~2.1]{GraySteinberg1},
if $Y,Z$ have $M$-finite $n$-skeleton and $X$ has $M$-finite $(n-1)$-skeleton (whence $X\times I$ has $M$-finite $n$-skeleton), then $U$ has $M$-finite $n$-skeleton.

The homotopy pushout construction is functorial with respect to commutative diagrams
\[
\begin{tikzcd} Y\ar{r}{f_1}\ar{rd}{r}\ar{d}[swap]{f_2}& X_1\ar{rd}{s} &\\
                 X_2\ar{rd}[swap]{t}      &Y'\ar{r}{g_1}\ar{d}{g_2} & X_1'\\
                 & X_2' & &\end{tikzcd}
\]
Moreover, if $r,s,t$ are homotopy equivalences, then it is well known that the induced mapping of homotopy pushouts is a homotopy equivalence; see for example~\cite[Theorem~4.2.1]{tomDieckBook}, or~\cite[page 19]{DwyerHennBook}
where it is observed that homotopy colimits have the strong homotopy equivalence property.

For the reader's convenience, we shall prove a special case of this fact that will be crucial in what follows. Recall that if $Y$ is a space,  the \emph{suspension} of $Y$ is the space $\Sigma Y=Y\times I/(Y\times \{0\}\cup Y\times \{1\})$. If $Y$ is contractible, then the mapping $\Sigma Y\to I$ induced by the projection $Y\times I\to I$ is a homotopy equivalence.

\begin{Lemma}\label{l:pushout.graph}
Let $M$ be a monoid and $X_1,X_2,Y$ locally path connected $M$-spaces.  Assume that the natural mappings $r_i\colon X_i\to \pi_0(X_i)$, for $i=1,2$, and $r\colon Y\to \pi_0(Y)$ are homotopy equivalences (where the set of path components is given the discrete topology).  Let $f_i\colon Y\to X_i$ be continuous mappings, for $i=1,2$, and let $Z$ be the homotopy pushout of $X_1,X_2$ along $Y$, which is naturally an $M$-space.  Let $\Gamma$ be the $M$-graph with vertex set $\pi_0(X_1)\coprod \pi_0(X_2)$ and edge set $\pi_0(Y)$ where the edge corresponding to $C\in \pi_0(Y)$ connects the component of $f_1(C)$ to the component of $f_2(C)$; this is the homotopy pushout of $\pi_0(X_1)$ and $\pi_0(X_2)$ along $\pi_0(Y)$.  Then the natural $M$-equivariant mapping $h\colon Z\to \Gamma$ is a homotopy equivalence.
\end{Lemma}
\begin{proof}
The mapping $h$ takes an element of $X_i$ to its path component and an element $(y,t)\in Y\times I$ to $(C,t)$ where $C$ is the component of $y$.  This is well defined, by construction of the homotopy pushout, and is $M$-equivariant.
As the connected components of $X_i$, for $i=1,2$, are disjoint and contractible subcomplexes, $Z$ is homotopy equivalent to the space obtained by contracting each of these subcomplexes to a point.  Then $Z$ has the homotopy type of the CW complex obtained by adjunction of $\coprod_{C\in \pi_0(Y)} \Sigma C$ to the discrete set $\pi_0(X_1)\coprod \pi_0(X_2)$ where $\Sigma C$ is attached via the mapping sending $(y,0)$ to the component of $f_1(C)$ and $(y,1)$ to the component of $f_2(C)$.  Since the mapping $\Sigma C\to I$ induced by the projection $C\times I\to I$ is a homotopy equivalence by contractibility of $C$, it follows that $h$ is a homotopy equivalence.  This completes the proof.
\end{proof}

We now prove some preservation results for amalgamated free products.   We shall apply the observation in Remark~\ref{r:tensor.with.free} without comment.

\begin{Thm}\label{t:bass.serre.free}
Let $[M_1,M_2;W]$ be an amalgam of monoids such that $M_1,M_2$ are free as right $W$-sets.  If $M_1,M_2$ are of type left-$\F_n$ and $W$ is of type left-$\F_{n-1}$, then $M_1\ast_W M_2$ is of type left-$\F_n$.
\end{Thm}
\begin{proof}
Let $X_i$ be an equivariant classifying space for $M_i$ with $M_i$-finite $n$-skeleton, for $i=1,2$, and let $Y$ be an equivariant classifying space for $W$ with $W$-finite $(n-1)$-skeleton.  By
\cite[Lemma~6.2]{GraySteinberg1}
and the cellular approximation theorem \cite[Theorem 2.8]{GraySteinberg1}, we can find $W$-equivariant cellular mappings $f_i\colon Y\to X_i$, for $i=1,2$.  Let $L=M_1\ast_W M_2$.  By McDuff~\cite{McDuff1979}, $L$ is a free right $M_i$-set, for $i=1,2$, and a free right $W$-set.  Then $X_i'=L\otimes_{M_i} X_i$, for $i=1,2$, is a projective $L$-CW complex with $L$-finite $n$-skeleton and $Y'=L\otimes_W Y$ is a projective $L$-CW complex with $L$-finite $(n-1)$-skeleton by Proposition~\ref{c:base.change.cw}.  Let $\tilde{f}_i\colon Y'\to X_i'$ be the map induced by $f_i$, for $i=1,2$, and let $Z$ be the homotopy pushout of $\tilde{f}_1,\tilde{f}_2$.  It is a projective $L$-CW complex.   We claim that $Z$ is an equivariant classifying space for $L$.  Note that $Z$ has an $L$-finite $n$-skeleton by construction.

Our goal is to show that $Z$ is homotopy equivalent to the Bass--Serre tree $T$. By
\cite[Proposition~3.4]{GraySteinberg1},
we have that $\pi_0(X_i')\cong L\otimes_{M_i} \pi_0(X_i)\cong L/M_i$ and $\pi_0(Y')\cong L\otimes_W \pi_0(Y)\cong L/W$ and $f_i$ induces the natural mapping $L/W\to L/M_i$ under these identifications, for $i=1,2$.    As $X_i'\cong L/M_i\times X_i$ and $Y'\cong L/W\times Y$ (by freeness of $L$ as a right $K$-set for $K=M_1,M_2,W$) and $X_i$, for $i=1,2$, and $Y$ are contractible, the projections $X_i'\to \pi_0(X_i')$, for $i=1,2$, and $Y'\to \pi_0(Y')$ are homotopy equivalences.  It follows that $Z$ is homotopy equivalent to $T$, by Lemma~\ref{l:pushout.graph}, and hence contractible.  This completes the proof.
\end{proof}

Note that we do not assume that the monoids $M_1$ and $M_2$ are finitely generated, or finitely presented, in the above result. Recall that a monoid can be of type left-$\F_2$ without being finitely presented, and can be of type left-$\F_1$ without being finitely generated; see~\cite[Section~6]{GraySteinberg1}.
The hypotheses of Theorem~\ref{t:bass.serre.free} hold if $W$ is trivial or if $M_1,M_2$ are left cancellative and $W$ is a group. As another example, if we consider $\mathbb N$, then, for any $k>0$, $\mathbb N$ is a free $k\mathbb N$-set with basis $\{0,1,\ldots, k-1\}$. Since $k\mathbb N\cong \mathbb N$, it follows from Theorem~\ref{t:bass.serre.free} that $\mathbb N\ast_{k\mathbb N=m\mathbb N} \mathbb N$ is of type left-$\F_{\infty}$, as $\mathbb N$ is of type left-$\F_{\infty}$, for any $k,m>0$. As a special case of Theorem~\ref{t:bass.serre.free} we obtain the following result as a corollary.  

\begin{Cor}\label{c:free.prod.fp}
A free product $M\ast N$  of monoids of type left-$\F_n$ is of type left-$\F_n$. If $M,N$ are finitely presented monoids, then $M\ast N$ is of type left-$\Fn$ if and only if $M$ and $N$ both are of type left-$\Fn$.
\end{Cor}
\begin{proof}
If $M$ and $N$ are of type left-$\F_n$, then $M \ast N$ is of type left-$\F_n$ by Theorem~\ref{t:bass.serre.free} as $M,N$ are free $\{1\}$-sets.  Conversely, if $M,N$ are finitely presented, then so is $M\ast N$ and hence left $\Fn$ is equivalent to left-$\FPn$ for these monoids.  A result of Pride~\cite{Pride2006} says that a retract of a left-$\FPn$ monoid is left-$\FPn$.  As $M,N$ are retracts of $M\ast N$, the converse follows.
\end{proof}

The fact that for finitely presented monoids $M,N$ of type left-$\FP_n$, the free product $M\ast N$ is of type left-$\FP_n$ was first proved in~\cite[Theorem~5.5]{CremannsOtto1998}.

The following corollary is classical.

\begin{Cor}
If $[G_1,G_2;H]$ is an amalgam of groups with $G_1,G_2$ of type left-$\F_n$ and $H$ of type left-$\F_{n-1}$, then $G_1\ast_H G_2$ is of type left $\F_n$.
\end{Cor}
\begin{proof}
Since $G_1,G_2$ are free left $H$-sets, this follows from Theorem~\ref{t:bass.serre.free}.
\end{proof}

The homotopy pushout construction in the proof of Theorem~\ref{t:bass.serre.free} also serves to establish the following.

\begin{Thm}\label{t:amalg.cd}
  Let $[M_1,M_2;W]$ be an amalgam of monoids such that $M_1,M_2$ are
  free as right $W$-sets.  Suppose that $d_i$ is the left geometric
  dimension of $M_i$, for $i=1,2$, and $d$ is the left geometric
  dimension of $W$. Then the left geometric dimension of
  $M_1\ast_W M_2$ is bounded above by $\max\{d_1,d_2,d+1\}$.
\end{Thm}

\begin{Cor}\label{c:free.prod.geom}
  Let $M$ and $N$ be monoids of left geometric dimension at most $n$.
  Then $M\ast N$ has left geometric dimension at most $\max\{n,1\}$.
\end{Cor}

We now wish to prove a homological analogue of
Theorem~\ref{t:bass.serre.free}.

\begin{Thm}\label{t:bass.serre.flat}
Let $[M_1,M_2;W]$ be an amalgam of monoids such that $\mathbb ZL$ is flat as a right $\mathbb ZM_1$-, $\mathbb ZM_2$- and $\mathbb ZW$-module, where $L=M_1\ast_W M_2$.  If $M_1,M_2$ are of type left-$\FP_n$ and $W$ is of type left-$\FP_{n-1}$, then $M_1\ast_W M_2$ is of type left-$\FP_n$.
\end{Thm}
\begin{proof}
By Lemma~\ref{l:flat.base} and the hypotheses, we deduce that $\mathbb ZL\otimes_{\mathbb ZM_i} \mathbb Z$ is of type $\FP_n$, for $i=1,2$, and $\mathbb ZL\otimes_{\mathbb ZW}\mathbb Z$ is of type $\FP_{n-1}$.  The result now follows by applying Corollary~\ref{c:fp.resolved} to the exact sequence in Corollary~\ref{c:exact.mv.begin}.
\end{proof}

\begin{remark}
It is reasonable to consider whether it might be possible to weaken the hypothesis of Theorem~\ref{t:bass.serre.flat} to just assuming that $\mathbb ZM_1$ and $\mathbb ZM_2$ are flat as $\mathbb ZW$-modules.
In~\cite[Lemma~5.2(a)]{Fiedorowicz1984}, Fiedorowicz claims that if $[M_1,M_2;W]$ is an amalgam of monoids such that $\mathbb ZM_1$ and $\mathbb ZM_2$ are flat as left $\mathbb ZW$-modules, then $\mathbb ZL$ (where $L=M_1\ast_W M_2$) is flat as a left $\mathbb ZM_i$-module, for $i=1,2$, and as a left $\mathbb ZW$-module. Unfortunately, his result is not correct. The following counterexample to~\cite[Lemma~5.2(a)]{Fiedorowicz1984} is due to Tyler Lawson (see~\cite{MOExample}), whom we thank for allowing us to reproduce it.
Let
\[
M_1 = \langle a, a^{-1} \mid aa^{-1} = 1, \; a^{-1}a = 1 \rangle, \quad
W = \{ b \}^*, \quad \mbox{and} \quad
M_2 = \{ c,d \}^*.
\]
So $M_1$ is isomorphic to the infinite cyclic group, and $W$ and $M_2$ are the free monoids of ranks $1$ and $2$, respectively. Let $f_1\colon W \rightarrow M_1$ be the homomorphism which maps $b \mapsto a$, let $f_2\colon W \rightarrow M_2$ be the homomorphism which maps $b \mapsto c$, and let $L$ be the monoid amalgam $[M_1,M_2;W]$ with respect to the embeddings $f_1$ and $f_2$.
Then $L$ is isomorphic to the monoid with presentation
\[
\langle a,a^{-1}, d \mid aa^{-1} = 1, a^{-1}a = 1 \rangle,
\]
that is, to $\mathbb Z\ast \{d\}^*$.  

As the commutative ring $\mathbb ZM_1$ is a localization of $\mathbb ZW$, it is clearly flat as a left $\mathbb ZW$-module.  Since $W$ is a free factor in $M_2$, we have that $M_2$ is a free left $W$-set and hence $\mathbb ZM_2$ is a free left $\mathbb ZW$-module (and thus flat). On the other hand, $\mathbb{Z}L$ is not flat as a left $\mathbb{Z}M_2$-module. This may be shown by considering the exact sequence of $\mathbb{Z}M_2$-modules
\[
0 \rightarrow \mathbb{Z}M_2 \oplus \mathbb{Z}M_2
\rightarrow \mathbb{Z}M_2
\rightarrow \mathbb{Z}
\rightarrow 0,
\]
where the first map sends $(u,v)$ to $uc + vd$, and the second sends $c$ and $d$ to zero. Here $\mathbb{Z}$ is made a left $\mathbb{Z}M_2$-module by having $c$ and $d$ annihilate it rather than via the trivial module structure. Tensoring this sequence over $\mathbb{Z}M_2$ on the left by $\mathbb{Z}L$ gives the sequence
\[
0 \rightarrow \mathbb{Z}L \oplus \mathbb{Z}L
\rightarrow \mathbb{Z}L
\rightarrow 0
\rightarrow 0,
\]
which is not left exact since the first factor of the direct sum is taken isomorphically to the middle term by invertibility of $a$. Hence $\mathbb{Z}L$ is not flat as a $\mathbb{Z}M_2$-module.  A nearly identical proof was given by Bergman to show that universal localization does not preserve flatness in the non-commutative setting~\cite[Page~70]{Bergman74}.

Since~\cite[Lemma~5.2(a)]{Fiedorowicz1984} does not hold, it cannot be used to weaken the
hypothesis of Theorem~\ref{t:bass.serre.flat} to assuming only that $\mathbb ZM_1$ and $\mathbb ZM_2$ are flat as $\mathbb ZW$-modules. Similarly~\cite[Lemma~5.2(a)]{Fiedorowicz1984} cannot be used to weaken the hypotheses of any of Theorems~\ref{t:bass.serre.flat.cd}, \ref{t:bass.serre.flat.bi} or \ref{t:bass.serre.flat.bi2}.
\end{remark}

It follows from results of McDuff~\cite{McDuff1979} that the hypotheses of Theorem~\ref{t:bass.serre.free} are satisfied when $M_1$ and $M_2$ are free as $W$-sets which gives the following corollary.

\begin{Cor}\label{t:bass.serre.free.hom}
Let $[M_1,M_2;W]$ be an amalgam of monoids such that $M_1,M_2$ are free as right $W$-sets.  If $M_1,M_2$ are of type left-$\FP_n$ and $W$ is of type left-$\FP_{n-1}$, then $M_1\ast_W M_2$ is of type left-$\FP_n$.
\end{Cor}

Corollary~\ref{t:bass.serre.free.hom} applies, in particular, when $W$ is trivial.  Thus we obtain the following improvement on~\cite[Theorem~5.5]{CremannsOtto1998} in which we do not need to assume the factors are finitely presented.

\begin{Cor}\label{c:free.prod.hom}
Let $M_1,M_2$ be monoids of type left-$\FP_n$.  Then $M_1\ast M_2$ is of type left-$\FP_n$.
\end{Cor}

\begin{Thm}\label{t:bass.serre.flat.cd}
Let $[M_1,M_2;W]$ be an amalgam of monoids such that $\mathbb ZL$, where $L=M_1\ast_W M_2$, is flat as a right $\mathbb ZM_1$-, $\mathbb ZM_2$- and $\mathbb ZW$-module. If $M_1,M_2$ have left cohomological dimension at most $d$ and $W$ has left cohomological dimension at most $d-1$, then $M_1\ast_W M_2$ has left cohomological dimension at most $d$.
\end{Thm}
\begin{proof}
By Lemma~\ref{l:flat.base} and the hypotheses, we deduce that $\mathbb ZL\otimes_{\mathbb ZM_i} \mathbb Z$ is of cohomological dimension at most $d$, for $i=1,2$, and $\mathbb ZL\otimes_{\mathbb ZA}\mathbb Z$ is of cohomological dimension $d-1$.  We deduce the theorem by applying Corollary~\ref{c:fp.resolved} to the exact sequence in Corollary~\ref{c:exact.mv.begin}.
\end{proof}

Again, combining this with results of McDuff~\cite{McDuff1979} gives the following.

\begin{Cor}\label{t:amalg.cd.cohom}
  Let $[M_1,M_2;W]$ be an amalgam of monoids such that $M_1,M_2$ are
  free as right $W$-sets.  Suppose that $d_i$ is the left cohomological
  dimension of $M_i$, for $i=1,2$, and $d$ is the left cohomological
  dimension of $W$. Then the left cohomological dimension of
  $M_1\ast_W M_2$ is bounded above by $\max\{d_1,d_2,d+1\}$.
\end{Cor}

\subsection{The two-sided setting}
We need some preliminary properties of tensor products before investigating amalgams in the two-sided context.

\begin{Prop}\label{p:tensor.id}
If $f\colon M\to N$ is a monoid homomorphism, then there is an $N\times N^{op}$ isomorphism $F\colon N\otimes_M M\otimes_M N\to N\otimes_M N$ defined by $F(n\otimes m\otimes n')=nm\otimes n'$.
\end{Prop}
\begin{proof}
The mapping $h\colon N\times M\times N\to N\otimes_M N$ given by $(n,m,n')\mapsto nm\otimes n'$ is $N\times N^{op}$-equivariant and satisfies $(nm',m,m''n')\mapsto nm'm\otimes m''n'=nm'mm''\otimes n'=h(n,m'mm'',n')$ and so the mapping $F$ is well defined.  The mapping $k\colon N\times N\to N\otimes_M M\otimes_M N$ given by $(n,n')\mapsto n\otimes 1\otimes n$ satisfies $k(nm,n')=nm\otimes 1\otimes n' = n\otimes m\otimes n'=n\otimes 1\otimes mn'=k(n,mn')$ for $m\in M$ and hence induces a mapping $N\otimes_M N\to N\otimes_M M\otimes_M N$.  Clearly, $h$ and $k$ induce inverse mappings as $nm\otimes 1\otimes n'=n\otimes m\otimes n'$ for $m\in M$.
\end{proof}

The next proposition will frequently be used to decongest notation.

\

\begin{Prop}\label{p:bi.tensor.iso}
Let $A$ be right $M$-set, $B$ a left $M$-set and $C$ a left $M\times M^{op}$-set.  Then $A\otimes_M C\otimes_M B$ is naturally isomorphic to $(A\times B)\otimes_{M\times M^{op}} C$ in the category of sets where we view $A\times B$ as a right $M\times M^{op}$-set via the action $(a,b)(m,m') = (am,m'b)$.
\end{Prop}
\begin{proof}
Define $f\colon A\times C\times B\to (A\times B)\otimes_{M\times M^{op}} C$ by $f(a,c,b)=(a,b)\otimes c$. Then $f(am,c,m'b) = (am,m'b)\otimes c=(a,b)\otimes mcm'$ and so $f$ induces a well-defined mapping $A\otimes_M C\otimes_M B\to (A\times B)\otimes_{M\times M^{op}} C$.  Define $g\colon A\times B\times C\to A\otimes_M C\otimes_M B$ by $g(a,b,c) = a\otimes c\otimes b$.  Then $g(am,m'b,c) = am\otimes c\otimes m'b = a\otimes mcm'\otimes b= g(a,b,mcm')$ and so $g$ induces a well-defined mapping $(A\times B)\otimes_{M\times M^{op}} C\to A\otimes_M C\otimes_M B$.  The maps induced by $f$ and $g$ are clearly mutually inverse and natural in $A,B,C$.
\end{proof}

\begin{Rmk}\label{r:same.arg}
A nearly identical proof shows that if $A$ is a right $\mathbb ZM$-module, $B$ is a left $\mathbb ZM$-module and $C$ is an $\mathbb ZM$-bimodule, then we have that $A\otimes_{\mathbb ZM} C\otimes_{\mathbb ZM}B\cong (A\otimes B)\otimes_{\mathbb ZM\otimes \mathbb ZM^{op}} C$ as abelian groups and the isomorphism is natural.
\end{Rmk}

\begin{Prop}\label{p:tensor.free}
Suppose that $A$ is a free right $M$-set, $B$ is a free left $M$-set and $C$ is an $M$-$M$-biset.  Then $A\otimes_M C\otimes_M B$ is naturally isomorphic to $A/M\times C\times M\backslash B$ in the category of sets.
\end{Prop}
\begin{proof}
By freeness, $A\otimes_M C\cong A/M\times C$ via $a\otimes c\mapsto ([a],c)$ where $[a]$ is the class of $a$ and, moreover, this is a right $M$-set isomorphism.  Therefore, $A\otimes_M C\otimes_M B\cong (A/M\times C)\otimes_M B\cong A/M\times C\times M\backslash B$ because $B$ is a free left $M$-set on $M\backslash B$.  The isomorphism is clearly natural in $A,B,C$.
\end{proof}

We now wish to consider a pushout diagram \eqref{eq:pushout.mon} in the bimodule setting.  Let us assume that $f_i\colon W\to M_i$ is the homomorphism in the diagram, for $i=1,2$, and we continue to use $L$ to denote the pushout.  Let us proceed to define a forest $T$. The vertex set of $T$ will be
\[V=(L\otimes_{M_1} L)\coprod (L\otimes_{M_2} L)\] and the edge set will be
\[E=L\otimes_W L.\]  We shall write $[x,y]_{K}$ for the tensor $x\otimes y$ in $L\otimes_{K} L$ for $K=M_1,M_2,W$.  The edge $[x,y]_W$ will connect $[x,y]_{M_1}$ to $[x,y]_{M_2}$, and we think of it as oriented in this direction.  Note that $T$ is an $L\times L^{op}$-graph.  Note that $[x,y]_K\mapsto xy$ is well defined for any of $K=M_1,M_2,W$.

\begin{Lemma}\label{l:bass.serre.com}
There is an $L\times L^{op}$-equivariant isomorphism $\pi_0(T)\to L$ induced by the multiplication map on vertices.
\end{Lemma}
\begin{proof}
As an edge $[x,y]_W$ connects $[x,y]_{M_1}$ to $[x,y]_{M_2}$, we have that multiplication $[x,y]_{M_i}\mapsto xy$ on vertices induces an $L\times L^{op}$-equivariant surjective mapping $\pi_0(T)\to L$.  To prove the injectivity, we first claim that $[1,x]_{M_1}$ is connected by an edge path to $[x,1]_{M_1}$ for all $x\in L$ by induction on the length of $x$.  If $x=1$, there is nothing to prove.  So assume the claim for length $k$ and let $x=x_1\cdots x_{k+1}$ with $x_i\in M_1\cup M_2$ (again abusing notation as $M_i$ need not embed in $L$).  Let $p$ be a path from $[1,x_2\cdots x_{k+1}]_{M_1}$ to $[x_2\cdots x_{k+1},1]_{M_1}$.  Then $x_1p_1$ is a path from $[x_1,x_2\cdots x_{k+1}]_{M_1}$ to $[x,1]_{M_1}$.  If $x_1\in M_1$, then $[x_1,x_2\cdots x_{k+1}]_{M_1}=[1,x]_{M_1}$ and we are done.  If $x_1\in M_2$, then $[x_1,x_2\cdots x_{k+1}]_W$ is an edge between $[x_1,x_2\cdots x_{k+1}]_{M_1}$ and $[1,x]_{M_2}$.  But $[1,x]_{W}$ is an edge from $[1,x]_{M_1}$ to $[1,x]_{M_2}$ and so we are again done in this case.

If $x=x_1x_2$ with $x_1,x_2\in L$, there is a path $p$ from $[1,x_1]_{M_1}$ to $[x_1,1]_{M_1}$ by the above claim. Then $px_2$ is a path from $[1,x]_{M_1}$ to $[x_1,x_2]_{M_1}$.  Thus any two vertices $[u,v]_{M_1}$ and $[u',v']_{M_1}$ with $uv=u'v'$ are connected in $T$.  But $[u,v]_W$ connects $[u,v]_{M_2}$ to $[u,v]_{M_1}$ and hence any two vertices $[u,v]_{M_i}$ and $[u',v']_{M_j}$ with $uv=u'v'$ are connected for all $i,j\in \{1,2\}$.  This completes the proof.
\end{proof}

Next we prove that $T$ is a forest.   Note that $\mathbb ZE$ is a $\mathbb ZL$-bimodule.

If $A$ is a bimodule over a monoid ring $\mathbb ZK$ then we can form the two-sided semidirect product $A\bowtie K$, of the abelian group $A$ and the monoid $K$, with elements $A \times K$ and multiplication given by
\[
(a,k)(a',k') = (ak' + ka', kk').
\]
A \emph{splitting} $\sigma$ of the projection $\pi\colon  A \bowtie K \rightarrow K$  is a monoid homomorphism $\sigma\colon  K \rightarrow A \bowtie K$ such that $\pi(\sigma(k))=k$ for all $k \in K$. A mapping $d\colon K\to A$ is a \emph{derivation} if
\[d(kk') =kd(k') + d(k)k'\] for all $k, k' \in K$.  A derivation in \emph{inner} if $d(k) = ka-ak$ for some $a\in A$.  Derivations correspond to splittings of the two-sided semidirect product projection $A\bowtie K\to K$, each splitting being of the form $k\mapsto (d(k),k)$ with $d$ a derivation.

\begin{Lemma}\label{l:forest}
The graph $T$ is a forest.
\end{Lemma}
\begin{proof}
A graph with vertex set $V$ and edge set $E$ is a forest if and only if the cellular boundary map $\partial\colon \mathbb ZE\to \mathbb ZV$ is injective.  We again use derivations to construct a left inverse to $\partial$.  As usual, we identify elements of $M_1$, $M_2$ and $W$ with their images in $L$ (abusing notation).

Define $\p_1\colon M_1\to \mathbb ZE\bowtie L$ by $\p_1(m_1)=(0,m_1)$; this is clearly a homomorphism.  Next define $\p_2\colon M_2\to \mathbb ZE\bowtie L$ by $\p_2(m_2) = ([1,m_2]_W-[m_2,1]_W,m_2)$.  Note that $m_2\mapsto [1,m_2]_W-[m_2,1]_W$ is the inner derivation of the $\mathbb ZM_2$-bimodule $\mathbb ZE$ associated to the element $-[1,1]_W$ and hence $\p_2$ is a homomorphism.  If $w\in W$, then \[\p_2f_2(w) = ([1,w]_W-[w,1]_W,w) = (0,w)=\p_1f_1(w)\] as $[1,w]_W=[w,1]_W$ for $w\in W$.  Therefore, there is a homomorphism $\p\colon L\to \mathbb ZE\bowtie L$ extending $\p_1,\p_2$, which is a splitting of the projection by construction.  Thus $\p(x)=(d(x),x)$ for some derivation $d\colon L\to \mathbb ZE$ satisfying $d(m_1)=0$ for $m_1\in M_1$ and $d(m_2) = [1,m_2]_W-[m_2,1]_W$ for $m_2\in M_2$.

We now define $\beta\colon \mathbb ZV\to \mathbb ZE$ by $\beta([x,y]_{M_1}) = d(x)y$ and $\beta([x,y]_{M_2}) = d(x)y+[x,y]_W$.  To show that this is well defined, we need that if $m_1\in M_1$, then $[xm_1,y]_{M_1}$ and $[x,m_1y]_{M_1}$ are sent to the same element and if $m_2\in M_2$, then $[xm_2,y]_{M_2}$ and $[x,m_2y]_{M_2}$ are sent to the same element.  But $d(xm_1)y = xd(m_1)y+d(x)m_1y = d(x)m_1y$ because $d(m_1)=0$.  Also, we compute
\begin{align*}
d(xm_2)y+[xm_2,y]_W&=xd(m_2)y+d(x)m_2y+[xm_2,y]_W\\ &=x([1,m_2]_W-[m_2,1]_W)y+d(x)m_2y+[xm_2,y]_W\\ &= d(x)m_2y+[x,m_2y]_W.
\end{align*}

We then obtain \[\beta\partial([x,y]_W)=\beta([x,y]_{M_2})-\beta([x,y]_{M_1}) = d(x)y+[x,y]_W-d(x)y=[x,y]_W.\]  Thus $\beta\partial=1_{\mathbb ZE}$ and hence $\partial$ is injective.  This completes the proof that $T$ is a forest.
\end{proof}

We call $T$ the \emph{Bass--Serre forest} of the pushout.    Since $H_0(T)\cong \mathbb Z\pi_0(T)\cong \mathbb ZL$ as an $L\times L^{op}$-bimodule (by Lemma~\ref{l:bass.serre.com}), Lemma~\ref{l:forest} has the following reinterpretation.

\begin{Cor}\label{c:exact.bimod.forest}
There is an exact sequence of $L\times L^{op}$-modules
\[0\longrightarrow \mathbb ZL\otimes_{\mathbb ZW} \mathbb ZL\longrightarrow (\mathbb ZL\otimes_{\mathbb ZM_1}\mathbb ZL)\oplus (\mathbb ZL\otimes_{\mathbb ZM_2} \mathbb ZL)\longrightarrow \mathbb ZL\longrightarrow 0\]
 where  $L=M_1\ast_{W} M_2$ is the pushout.
\end{Cor}
\begin{proof}
This follows by consideration of the cellular chain complex of the forest $T$ and using that $\mathbb ZV/\partial\mathbb ZE=H_0(T)\cong \mathbb ZL$, as observed before the corollary.
\end{proof}

\begin{Thm}\label{t:bass.serre.free.bi}
Let $[M_1,M_2;W]$ be an amalgam of monoids such that $M_1,M_2$ are free as both left and right $W$-sets.  If $M_1,M_2$ are of type bi-$\F_n$ and $W$ is of type bi-$\F_{n-1}$, then $M_1\ast_W M_2$ is of type bi-$\F_n$.
\end{Thm}
\begin{proof}
Let $X_i$ be a bi-equivariant classifying space for $M_i$ with $M_i \times M_i^{op}$-finite $n$-skeleton, for $i=1,2$, and $Y$ a bi-equivariant classifying space for $W$ with $W \times W^{op}$-finite $(n-1)$-skeleton.  Fix bi-equivariant isomorphisms $r_i\colon M_i\to \pi_0(X_i)$ and $r\colon W\to \pi_0(Y)$.   By
\cite[Lemma~7.1]{GraySteinberg1}
and the cellular approximation theorem  \cite[Theorem 2.8]{GraySteinberg1}, we can find $W\times W^{op}$-equivariant cellular mappings $f_i\colon Y\to X_i$, for $i=1,2$, such that the composition of $r$ with the composition of the mapping induced by $f_i$ with $r_i^{-1}$ is the inclusion, for $i=1,2$.  Let $L=M_1\ast_W M_2$.  By McDuff~\cite{McDuff1979}, $L$ is a free as both a left and a right $M_i$-set, for $i=1,2$, and as a left and right $W$-set.

For $i=1,2$, $X_i'=L\otimes_{M_i} X_i\otimes_{M_i} L\cong (L\times L^{op})\otimes_{L\times L^{op}} X_i$ (the isomorphism by Proposition~\ref{p:bi.tensor.iso}) is a projective $L\times L^{op}$-CW complex with $L\times L^{op}$-finite $n$-skeleton and $Y'=L\otimes_W Y\otimes_W L\cong (L\times L^{op})\otimes_{L\times L^{op}} Y$ is a projective $L\times L^{op}$-CW complex with $L\times L^{op}$-finite $(n-1)$-skeleton by Proposition~\ref{c:base.change.cw}.  Let $F_i\colon Y'\to X_i'$ be the mapping induced by $f_i$, for $i=1,2$, and let $Z$ be the homotopy pushout of $F_1,F_2$; it is a projective $L\times L^{op}$-CW complex.   We claim that $Z$ is a bi-equivariant classifying space for $L$.  Note that $Z$ has an $L\times L^{op}$-finite $n$-skeleton by construction.

Our goal is to show that $Z$ is homotopy equivalent to the Bass--Serre forest $T$ via an $L\times L^{op}$-equivariant homotopy equivalence. By~\cite[Proposition~3.4]{GraySteinberg1}
 and Proposition~\ref{p:tensor.id} we have that $\pi_0(X_i')\cong L\otimes_{M_i} M\otimes_{M_i} L\cong L\otimes_{M_i} L$, for $i=1,2$, and $\pi_0(Y')\cong L\otimes_W W\otimes_W L\cong L\otimes_W L$ and, moreover, $F_i$ induces the natural mapping $L\otimes_W L\to L\otimes_{M_i} L$, for $i=1,2$ (by construction).  Thus, by Lemma~\ref{l:pushout.graph}, it suffices to show that the projections $X_i'\to \pi_0(X_i')$, for $i=1,2$, and $Y'\to \pi_0(Y)$ are homotopy equivalences.

 Since $L$ is free as a left and right $M_i$-set, for $i=1,2$, and as a left and right $W$-set, we have by Proposition~\ref{p:tensor.free} that $X_i'\cong L/M_i\times X_i\times M_i\backslash L$ (for $i=1,2$) and $Y'\cong L/W\times Y\times W\backslash L$.  As $X_1,X_2,Y$ are homotopy equivalent to their sets of path components via the canonical projection, we deduce that the projections to path components are, indeed, homotopy equivalences for $X_1',X_2',Y'$.  This completes the proof.
\end{proof}

The hypotheses of Theorem~\ref{t:bass.serre.free.bi}, of course, hold if $W$ is trivial. It also holds if we amalgamate two copies of $\mathbb N$ along cyclic submonoids.  So $\mathbb N\ast_{k\mathbb N=m\mathbb N}\mathbb N$ is of type bi-$\F_{\infty}$ for any $m,k>0$.

\begin{Cor}
A free product $M\ast N$  of monoids of type bi-$\F_n$ is of type bi-$\F_n$. If $M,N$ are finitely presented monoids, then $M\ast N$ is of type bi-$\FP_n$ if and only if $M$ and $N$ both are of type bi-$\FP_n$.
\end{Cor}
\begin{proof}
The first statement follows from Theorem~\ref{t:bass.serre.free.bi}.  The second follows from the equivalence of bi-$\F_n$ and bi-$\FP_n$ for finitely presented monoids and the result of Pride~\cite{Pride2006} that the class of monoids of type bi-$\FP_n$ is closed under retracts.
\end{proof}

The hypotheses of Theorem~\ref{t:bass.serre.free.bi} also hold if $M_1,M_2$ are cancellative and $W$ is a group.
The homotopy pushout construction in the proof of Theroem~\ref{t:bass.serre.free.bi} yields the following theorem.

\begin{Thm}\label{t:amalg.cd.bi}
Let $[M_1,M_2;W]$ be an amalgam of monoids such that $M_1,M_2$ are free as left and right $W$-sets.  Suppose that $d_i$ is the geometric dimension of $M_i$, for $i=1,2$ and $d$ is the geometric dimension of $W$. Then the geometric dimension of $M_1\ast_W M_2$ is bounded above by $\max\{d_1,d_2,d+1\}$.
\end{Thm}

Since only the trivial monoid has geometric dimension $0$, we obtain the following special case.

\begin{Cor}\label{c:free.prod.geom.bi}
Let $M$ and $N$ be monoids of geometric dimension at most $n$.  Then $M\ast N$ has geometric dimension at most $n$.
\end{Cor}

Next we wish to consider the homological analogue.

\begin{Prop}\label{p:flatness.bi}
Suppose that $A$ is a flat right $\mathbb ZM$-module and $B$ is a flat left $\mathbb ZM$-module.  Then $A\otimes B$ is a flat right $\mathbb ZM\otimes \mathbb ZM^{op}$-module (with respect to the structure $(a\otimes b)(m,m') = am\otimes m'b$).
\end{Prop}
\begin{proof}
If $0\longrightarrow J\longrightarrow K\longrightarrow L\longrightarrow 0$ is a short exact sequences of $M$-bimodules, then
$0\longrightarrow A\otimes_{\mathbb ZM} J\longrightarrow  A\otimes_{\mathbb ZM}K\longrightarrow A\otimes_{\mathbb ZM}L\longrightarrow 0$ is exact by flatness of $A$.  Therefore,
\[0\longrightarrow A\otimes_{\mathbb ZM} J\otimes_{\mathbb ZM} B\longrightarrow  A\otimes_{\mathbb ZM}K\otimes_{\mathbb ZM}B\longrightarrow A\otimes_{\mathbb ZM}L\otimes_{\mathbb ZM}B\longrightarrow 0\] is exact by flatness of $B$.  The result now follows by Remark~\ref{r:same.arg}.
\end{proof}

\begin{Thm}\label{t:bass.serre.flat.bi}
Let $[M_1,M_2;W]$ be an amalgam of monoids such that $\mathbb ZL$ is flat as both a left and right $\mathbb ZM_i$-module and $\mathbb ZW$-module, for $i=1,2$, where $L=M_1\ast_W M_2$.  If $M_1,M_2$ are of type bi-$\FP_n$ and $W$ is of type bi-$\FP_{n-1}$, then $M_1\ast_W M_2$ is of type bi-$\FP_n$.
\end{Thm}
\begin{proof}
Note that $\mathbb Z[L\times L^{op}]\cong \mathbb ZL\otimes \mathbb ZL^{op}$ is a flat right $\mathbb Z[M_i\times M_i^{op}]$-module, for $i=1,2$, and a flat right-$\mathbb Z[W\times W^{op}]$-module by Proposition~\ref{p:flatness.bi}. By Lemma~\ref{l:flat.base} and the hypotheses, we deduce that $\mathbb Z[L\times L^{op}]\otimes_{\mathbb Z[M_i\times M_i^{op}]} \mathbb ZM_i$ is of type $\FP_n$, for $i=1,2$, and $\mathbb Z[L\times L^{op}]\otimes_{\mathbb Z[W\times W^{op}]}\mathbb ZW$ is of type $\FP_{n-1}$.  The result now follows by applying Corollary~\ref{c:fp.resolved} to the exact sequence in Corollary~\ref{c:exact.bimod.forest}, in light of Proposition~\ref{p:tensor.id} and Proposition~\ref{p:bi.tensor.iso}.
\end{proof}

\begin{Thm}\label{t:bass.serre.flat.bi2}
Suppose that $[M_1,M_2;W]$ is an amalgam of monoids such that $M_i$ has Hochschild cohomological dimension at most  $d$, for $i=1,2$, $W$ has Hochschild cohomological  dimension  at most $d-1$, and $\mathbb ZL$ is flat as both a left and right $\mathbb ZM_i$-module and $\mathbb ZW$-module, for $i=1,2$, where $L=M_1\ast_W M_2$.   Then  $M_1\ast_W M_2$ has Hochschild cohomological dimension at most $d$.
\end{Thm}

As with the one-sided results, combining these results with results of McDuff~\cite{McDuff1979} gives the following corollaries.

\begin{Cor}\label{t:bass.serre.free.bi.hom}
Let $[M_1,M_2;W]$ be an amalgam of monoids such that $M_1,M_2$ are free as both left and right $W$-sets.  If $M_1,M_2$ are of type bi-$\FP_n$ and $W$ is of type bi-$\FP_{n-1}$, then $M_1\ast_W M_2$ is of type bi-$\FP_n$.  This applies, in particular, to free products.
\end{Cor}

\begin{Cor}\label{t:amalg.cd.bi.cohom}
Let $[M_1,M_2;W]$ be an amalgam of monoids such that $M_1,M_2$ are free as left and right $W$-sets.  Suppose that $d_i$ is the Hochschild cohomological dimension of $M_i$, for $i=1,2$ and $d$ is the Hochschild cohomological dimension of $W$. Then the Hochschild cohomological dimension of $M_1\ast_W M_2$ is bounded above by $\max\{d_1,d_2,d+1\}$.
\end{Cor}

We remark that the results of this section and the previous section have analogues for the amalgamation of a finite family of monoids over a common submonoid.

\section{HNN extensions} 
\label{sec_HNNOttoPride}
In this section we shall present several new theorems about the behaviour of homological and topological finiteness properties for HNN extensions of monoids.   
Several natural HNN extension definitions for monoids have arisen in the literature in different contexts. 

First in this section we consider a generalization of a construction of Otto and Pride, which they used to distinguish finite derivation type from finite homological type~\cite{Pride2004}. Let $M$ be a monoid, $A$ a submonoid and $\p\colon A\to M$ a homomorphism.  The free monoid generated by a set $A$ is denoted by $A^*$.  Then the \emph{Otto-Pride extension} of $M$ with base monoid $A$ is the quotient $L$ of the free product $M\ast \{t\}^*$ by the smallest congruence such that $at=t\p(a)$ for $a\in A$, i.e., $L=\langle M,t\mid at=t\p(a), a\in A\rangle$.  For example, if $A=M$ and $\p$ is the trivial homomorphism, then the Otto-Pride extension is the monoid $M\cup \ov M$ where $\ov M$ is an adjoined set of right zeroes in bijection with $M$.  Otto and Pride have considered Otto-Pride extensions of groups where $\varphi$ is injective, in~\cite{Pride2004} and~\cite{Pride2005}.

\subsection{The one-sided case}

The following model for $L$ will be useful for constructing normal forms and for proving flatness results.

\begin{Prop}\label{p:hnnlike.as.tensor}
View $M$ as a right $A$-set via right multiplication and as a left $A$-set via the action $a\odot m=\p(a)m$ for $a\in A$.  Then $L$ is isomorphic to the monoid
with underlying set $R=\coprod_{i=0}^{\infty}R_i$, where $R_0=M$ and $R_{i+1}=R_i\otimes_A M$,  and with multiplication defined by
\[(m_1\otimes \cdots\otimes m_k)(m_1'\otimes \cdots \otimes m'_\ell) = m_1\otimes \cdots \otimes m_{k-1}\otimes m_km_1'\otimes m_2'\otimes \cdots \otimes m'_{\ell}.\]  In particular, $M$ and $t^*$ embed in $L$ (where $t$ is identified with $1\otimes 1\in R_1$).
\end{Prop}
\begin{proof}
It is a straightforward exercise to verify that $R$ is a monoid with identity $1\in R_0=M$.   Define $f\colon M\cup \{t\}\to R$ by $f(m)=m$ and $f(t)=1\otimes 1$.  Then if $a\in A$, we have that $f(a)f(t)=a\otimes 1=1\otimes \p(a)=f(t)f(\p(a))$ and so $f$ induces a homomorphism $f\colon L\to R$.  Note that $f$ is surjective.  Indeed, $R_0$ is in the image of $f$ by construction.  Assume that $R_i$ is in the image of $f$ and let $m_1\otimes\cdots \otimes m_{i+1}\in R_i$.  If $f(x)=m_1\otimes \cdots\otimes m_i$ (by induction), then
$f(xtm_{i+1}) = m_1\otimes \cdots \otimes m_i\otimes m_{i+1}$.  Now define $g\colon R\to L$ by $g(m_1\otimes \cdots \otimes m_i) = m_1tm_2t\cdots tm_i$. It is easy to verify that this is well defined using the defining relations of $L$ and trivially $g$ is a homomorphism.  Now $gf(m)=m$ for $m\in M$ and $gf(t)=g(1\otimes 1)=t$.  Therefore, $gf=1_L$ and so $f$ is injective.  This concludes the proof that $f$ is an isomorphism.
\end{proof}

As a corollary, we can deduce a normal form theorem for $L$ if $M$ is free as a right $A$-set.

\begin{Cor}\label{c:normal.form}
Let $\p\colon A\to M$ be a homomorphism with $A$ a submonoid of $M$.  Let $L=\langle M,t\mid at=t\p(a), a\in A\rangle$ be the Otto-Pride extension.  Suppose that $M$ is a free right $A$-set with basis $C$ containing $1$.  Then every element of $M$ can be uniquely written in the form $c_0tc_1\cdots tc_ka$ with $k\geq 0$, $c_i\in C$ and $a\in A$.  Consequently, $L$ is free both as a right $M$-set and a right $A$-set.
\end{Cor}
\begin{proof}
Since $M$ is free as a right $A$-set on $C$, retaining the notation of Proposition~\ref{p:hnnlike.as.tensor}, we have that $R_i\cong C^{i+1}\times A$ via the mapping $(c_0,\ldots, c_i,a)\mapsto c_0\otimes c_1\otimes \cdots \otimes c_ia$. Composing this mapping with the isomorphism $g$ in the proof of Proposition~\ref{p:hnnlike.as.tensor} provides the desired normal form.  Clearly, $L$ is a free right $M$-set on the normal forms with $c_k=1=a$ and $L$ is a free right $A$-set on the normal forms with $a=1$.  This completes the proof.
\end{proof}

Note that if $M$ is left cancellative and $A$ is a group, then $M$ is a free right $A$-set.

\begin{Cor}\label{c:flat.hnnlike}
Let $M$ be a monoid, $A$ a submonoid and $\p\colon A\to M$ be a homomorphism.  Let $L=\langle M,t\mid at=t\p(a), a\in A\rangle$ be the Otto-Pride extension.  Suppose that $\mathbb ZM$ is flat as a right $\mathbb ZA$-module.  Then $\mathbb ZL$ is flat both as a right $\mathbb ZM$-module and a right $\mathbb ZA$-module.
\end{Cor}
\begin{proof}
Put $V_0=\mathbb ZM$ and $V_{i+1}=V_i\otimes_{\mathbb ZA}\mathbb ZM$.  Then by Proposition~\ref{p:hnnlike.as.tensor}, we have that as a right $\mathbb ZM$-module, $\mathbb ZL\cong \bigoplus_{i\geq 0} V_i$ so it suffices to show that $V_i$ is flat as both a right $\mathbb ZM$-module and a right $\mathbb ZA$-module.  We prove this by induction.  As $V_0$ is a free right $\mathbb ZM$-module and a flat $\mathbb ZA$-module, by assumption, this case is handled.  Assume that $V_i$ is flat both as a right $\mathbb ZM$-module and a right $\mathbb ZA$-module.  Let $h\colon U\to W$ be an injective homomorphism of $\mathbb ZM$-modules (respectively, $\mathbb ZA$-modules).  Then the induced mapping $\mathbb ZM\otimes_{\mathbb ZM} U\to \mathbb ZM\otimes_{\mathbb ZM} W$ (respectively,  $\mathbb ZM\otimes_{\mathbb ZA} U\to \mathbb ZM\otimes_{\mathbb ZA} W$) is injective since $\mathbb ZM$ is flat as a right module over both $\mathbb ZM$ and $\mathbb ZA$.  Then tensoring  these injective mappings on the left with $V_i$ over $\mathbb ZA$ results in an injective mapping by flatness of $V_i$.  Thus we see that $V_{i+1}$ is flat as a right $\mathbb ZM$-module and as a right $\mathbb ZA$-module.
\end{proof}

We now construct a Bass--Serre tree for Otto-Pride extensions.  Again fix a monoid $M$ together with a homomorphism $\p\colon A\to M$ from a submonoid $A$ and let $L$ be the Otto-Pride extension.  We define a graph $T$ with vertex set $V=L/M$ and edge set $E=L/A$.  An edge $[x]_A$ connects $[x]_M$ to $[xt]_M$ (oriented in this way), where $[x]_K$ denotes the class of $x$ in $L/K$.  This is well defined because if $a\in A$, then $[xa]_M=[x]_M$ and $[xat]_M = [xt\p(a)]_M=[xt]_M$.  Clearly, the left action of $L$ is by cellular mappings sending edges to edges and so $T$ is an $L$-graph. We aim to prove that $T$ is a tree.

\begin{Lemma}\label{l:bserre.op.conn}
The graph $T$ is connected.
\end{Lemma}
\begin{proof}
The monoid $L$ is generated by $M\cup \{t\}$.  The length of an element $x$ is its shortest expression as a product in these generators.  We prove by induction on length that there is a path from $[1]_M$ to $[x]_M$.  If $x=1$, there is nothing to prove.  Assume that $x=yz$ with $y\in M\cup \{t\}$ and $z$ of length one shorter.  Let $p$ be a path from $[1]_M$ to $[z]_M$. Then $yp$ is a path from $[y]_M$ to $[x]_M$.  If $y\in M$, then $[y]_M=[1]_M$ and we are done.  If $y=t$, then since $[1]_A$ connects $[1]_M$ with $[t]_M=[y]_M$ and so we are done in this case, as well.  It follows that $T$ is connected.
\end{proof}

Next we use derivations to prove that $T$ is a tree.

\begin{Lemma}\label{l:bserre.op.tree}
The graph $T$ is a tree.
\end{Lemma}
\begin{proof}
We prove that $\partial\colon \mathbb ZE\to \mathbb ZV$ is injective.  It will then follows that $T$ is a tree as it was already shown to be connected in Lemma~\ref{l:bserre.op.conn}.  Define $\gamma\colon M\cup\{t\}\to \mathbb ZE\rtimes L$ by $\gamma(m) = (0,m)$ for $m\in M$ and $\gamma(t)=([1]_A,t)$.  Then if $a\in A$, we have that $\gamma(a)\gamma(t) = (0,a)([1]_A,t) = ([a]_A,at)=([1]_A,t\p(a)) = ([1]_A,t)(0,\p(a))=\gamma(t)\gamma(\p(a))$.  Therefore, $\gamma$ extends to a homomorphism $\gamma\colon L\to \mathbb ZE\rtimes L$ splitting the semidirect product.  Thus $\gamma(x)=(d(x),x)$ for some derivation $d\colon L\to \mathbb ZE$ with $d(m)=0$ for $m\in M$ and $d(t)=[1]_A$.

Define $\beta\colon \mathbb ZV\to \mathbb ZE$ by $\beta([x]_M) = d(x)$.  This is well defined because if $m\in M$, then $d(xm)=xd(m)+d(x)=d(x)$ as $d(m)=0$.  Now we compute that $\beta\partial([x]_A) = \beta([xt]_M)-\beta([x]_M) = d(xt)-d(x) = xd(t)+d(x)-d(x)=x[1]_A=[x]_A$.  Therefore, $\beta\partial=1_{\mathbb ZE}$ and hence $\partial$ is injective.  We conclude that $T$ is a tree.
\end{proof}

We call $T$ the \emph{Bass--Serre tree} of the extension. Lemma~\ref{l:bserre.op.tree} can be restated in terms of exact sequences using that $\mathbb Z[L/K]\cong \mathbb ZL\otimes_{\mathbb ZK} \mathbb Z$ for $K=M,A$.

\begin{Cor}\label{c:exact.seq.OP}
There is an exact sequence
\[0\longrightarrow \mathbb ZL\otimes_{\mathbb ZA} \mathbb Z\longrightarrow\mathbb ZL\otimes_{\mathbb ZM}\mathbb Z\longrightarrow \mathbb Z\longrightarrow 0\]
of left $\mathbb ZL$-modules.
\end{Cor}

The analogue of the homotopy pushout that we shall need in this context is the homotopy coequalizer.  If $f,g\colon Y\to X$ are continuous mappings, then the \emph{homotopy coequalizer} $M(f,g)$ is the space obtained by gluing $Y\times I$ to $X$ via the mapping $h\colon Y\times \partial I\to X$ given by $h(y,0)=f(y)$ and $h(y,1)=g(y)$.  If $X$ and $Y$ are CW complexes and $f,g$ are cellular, then $M(f,g)$ is a CW complex.  If $X,Y$ are projective $M$-CW complexes and $f,g$ are $M$-equivariant and cellular, then $M(f,g)$ is a projective $M$-CW complex by
\cite[Lemma~2.1]{GraySteinberg1}.
Moreover, if $X$ has $M$-finite $n$-skeleton and $Y$ has $M$-finite $(n-1)$-skeleton, then $M(f,g)$ has $M$-finite $n$-skeleton.

Homotopy coequalizers like homotopy pushouts, are examples of homotopy colimits. If $f',g'\colon Y'\to X'$ are continuous mappings and  $r\colon Y\to Y'$ and $s\colon X\to X'$ are continuous such that
\[\begin{tikzcd}
Y \arrow[yshift=0.7ex]{r}{f} \arrow[yshift=-0.7ex,swap]{r}{g}\arrow{d}[swap]{r}
& X \arrow{d}{s} \\
Y'\arrow[yshift=0.7ex]{r}{f'} \arrow[yshift=-0.7ex]{r}[swap]{g'}&X'
\end{tikzcd}\] commutes, then there is an induced continuous mapping $t\colon M(f,g)\to M(f',g')$ (which will be $M$-equivariant if all spaces are $M$-spaces and all maps are $M$-equivariant).  Moreover, if $r,s$ are homotopy equivalences, then so is $t$; see~\cite[page 19]{DwyerHennBook}.  For example, the graph $T$ is the homotopy coequalizer of $i,j\colon L/A\to L/M$ given by $i([x]_A) = [x]_A$ and $j([x]_A) = [xt]_A$ (where these sets are viewed as discrete spaces).

\begin{Thm}\label{t:ottopride.one.side}
Let $M$ be a monoid, $A$ a submonoid and $\p\colon A\to M$ be a homomorphism.  Let $L=\langle M,t\mid at=t\p(a), a\in A\rangle$ be the Otto-Pride extension.  Suppose that $M$ is free as a right $A$-set.  If $M$ is of type left-$\F_n$ and $A$ is of type left-$\F_{n-1}$, then $L$ is of type left-$\F_n$.
\end{Thm}
\begin{proof}
Let $X$ be an equivariant classifying space for $M$ with $M$-finite $n$-skeleton and let $Y$ be an equivariant classifying space for $A$ with $A$-finite $(n-1)$-skeleton.  Using
\cite[Lemma~6.2]{GraySteinberg1}
and the
cellular approximation theorem \cite[Theorem 2.8]{GraySteinberg1}, we can find continuous cellular mappings $f,g\colon Y\to X$ such that $f(ay)=af(y)$ and $g(ay)=\p(a)g(y)$ for all $a\in A$ and $y\in Y$.  To construct $g$, we view $X$ as an $A$-space via the action $a\odot x=\p(a)x$ for $a\in A$.  Let $X'=L\otimes_M X$ and $Y'=L\otimes_A Y$.  These are projective $L$-CW complexes by Proposition~\ref{c:base.change.cw} and $X'$ has $L$-finite $n$-skeleton, $Y'$ has $L$-finite $(n-1)$-skeleton.

Let $F\colon Y'\to X'$ be the mapping induced by $f$ and define $G\colon Y'\to X'$ by $G(u\otimes y) =ut\otimes g(y)$. The latter is well defined since if $a\in A$, then $uat\otimes g(y)=ut\p(a)\otimes g(y) = ut\otimes \p(a)g(y)=ut\otimes g(ay)$. Clearly, $G$ is $L$-equivariant, continuous and cellular.  Let $Z=M(F,G)$ be the homotopy coequalizer.  Then $Z$ is a projective $L$-CW complex with $L$-finite $n$-skeleton.  We aim to show that $Z$ is homotopy equivalent to $T$ and hence contractible.

By~\cite[Proposition~3.4]{GraySteinberg1}
we have that $\pi_0(Y')\cong L\otimes_A \pi_0(Y)\cong L/A$ and $\pi_0(X')\cong L\otimes_M \pi_0(X)\cong L/M$ as $X,Y$ are connected. By construction $F$ and $G$ induce the mappings $[u]_A\mapsto [u]_M$ and $[u]_A\mapsto [ut]_M$, respectively, on path components under these identifications.   As the tree $T$ is the homotopy coequalizer of these two mappings, it suffices to show that the projections $X'\to \pi_0(X')$ and $Y'\to \pi_0(Y')$ are homotopy equivalences. Then $Z$ will be homotopy equivalent to $T$.

 Since $L$ is free as a right $M$-set and as a right $A$-set, we have that $X'\cong L/M\times X$ and $Y'\cong L/A\times Y$ as $L$-CW complexes.
  As $X$ and $Y$ are contractible and $L/M$ and $L/A$ are discrete, we deduce that the projections to connected components are homotopy equivalences in both cases.  This completes the proof.
\end{proof}

The proof of Theorem~\ref{t:ottopride.one.side} can be used to show that if $M$ is free as a right $A$-set, $M$ has left geometric dimension $d$ and $A$ has left geometric dimension $d'$, then $L$ has left geometric dimension at most $\max\{d,d'+1\}$.  The hypothesis  of Theorem~\ref{t:ottopride.one.side} applies if $M$ is left cancellative and $A$ is a group or if $M=\mathbb N$ and $A$ is a cyclic submonoid.

Next we prove the homological analogue of Theorem~\ref{t:ottopride.one.side} under the weaker assumption of flatness.

\begin{Thm}\label{t:ottopride.one.side.flat}
Let $M$ be a monoid and let $\p\colon A\to M$ be a homomorphism from a submonoid $A$ of $M$.  Let $L=\langle M,t\mid at=t\p(a), a\in A\rangle$ be the Otto-Pride extension.  Suppose that $\mathbb ZM$ is flat as a right $\mathbb ZA$-module.  If $M$ is of type left-$\FP_n$ and $A$ is of type left-$\FP_{n-1}$, then $L$ is of type left-$\FP_n$.
\end{Thm}
\begin{proof}
By Corollary~\ref{c:flat.hnnlike}, $\mathbb ZL$ is flat as a  right $\mathbb ZM$-module and as a  right $\mathbb ZA$-module. It follows from Lemma~\ref{l:flat.base} and the hypotheses that $\mathbb ZL\otimes_{\mathbb ZM} \mathbb Z$ is of type $\FP_n$ and $\mathbb ZL\otimes_{\mathbb ZA}\mathbb Z$ is of type $\FP_{n-1}$.  The result now follows by applying Corollary~\ref{c:fp.resolved} to the exact sequence in Corollary~\ref{c:exact.seq.OP}.
\end{proof}

One can prove similarly the following theorem.

\begin{Thm}\label{t:ottopride.one.side.flat.cd}
Let $M$ be a monoid and $\p\colon A\to M$ a homomorphism from a submonoid $A$ of $M$.  Let $L=\langle M,t\mid at=t\p(a), a\in A\rangle$ be the Otto-Pride extension.  Suppose that $\mathbb ZM$ is flat as a right $\mathbb ZA$-module.  If $M$ has left cohomological dimension at most $d$ and $A$ has left cohomological dimension at most $d-1$, then $L$ has left cohomological dimension at most $d$.
\end{Thm}

\subsection{The two-sided case}
It turns out that in the two-sided setting we shall need to consider Otto-Pride extensions corresponding to injective monoid homomorphisms $\p\colon A\to M$ from a submonoid $A$ of $M$ in order to make the construction left-right dual.  Putting $B=\p(A)$, we have that $B$ is isomorphic to $A$.  Otto and Pride considered the special case when $M$ and $A$ are groups (and hence so is $B$).  We shall call an Otto-Pride extension \emph{HNN-like} if $\p$ is injective. Let $L$ be the Otto-Pride extension. It is straightforward to check $L=\langle M,t\mid tb=\p\inv(b)t, b\in B\rangle$ and hence left/right duals of Proposition~\ref{p:hnnlike.as.tensor} and Corollary~\ref{c:normal.form} are valid with $B$ in the role of $A$ and using left sets instead of right sets.  Note that an HNN-like Otto-Pride extension of groups, which is the case considered by Otto and Pride, embeds as a submonoid of the corresponding group HNN extension (note that the Otto-Pride extension does not contain $t^{-1}$ and hence is a monoid, not a group).  Our results give geometric proofs of a number of the results of~\cite{Pride2004} and~\cite{Pride2005}.

 In what follows, we shall always view $L$ as a right $A$-set via left multiplication and as a left $A$-set via $a\odot x= \p(a)x$.  Therefore, we view $L\times L^{op}$ as a right $A\times A^{op}$-set via $(x,y)(a,a') = (xa,\p(a')y)$.

\begin{Prop}\label{p:basic.tensor.ids}
There is an isomorphism \[L\otimes_A L\cong L\otimes_A A\otimes_A L\cong (L\times L^{op})\otimes_{A\times A^{op}} A\] of left $L\times L^{op}$-sets.
\end{Prop}
\begin{proof}
The first isomorphism is given by $x\otimes y\mapsto x\otimes 1\otimes y$ with inverse $x\otimes a\otimes y\mapsto xa\otimes y$ (the reader should check that these are well defined and equivariant).  The second isomorphism sends $x\otimes a\otimes y$ to $(x,y)\otimes a$ with inverse mapping $(x,y)\otimes a$ to $x\otimes a\otimes y$.  The reader should again check that this is well defined and equivariant.
\end{proof}

We now associate a Bass--Serre forest $T$ to an HNN-like Otto-pride extension.  The vertex set of $T$ is $V=L\otimes_M L$ and the edge set is $E=L\otimes_A L$.  Again, we write $[x,y]_K$ for the tensor $x\otimes y$ of $L\otimes_K L$, for $K=M,A$.  With this notation, the edge $[x,y]_A$ connects $[x,ty]_M$ to $[xt,y]_M$ (which we think of as oriented in this way).  To check that this is well defined, observe that if $x,y\in L$ and $a\in A$, then $[xa,y]_A=[x,\p(a)y]_A$ and  $[xa,ty]_M = [x,aty]_M = [x,t\p(a)y]_M$ and $[xat,y]_M=[xt\p(a),y]_M = [xt,\p(a)y]_M$. By construction, $T$ is an $L\times L^{op}$-graph.

It is immediate from the definition of the incidences in $T$ that  the multiplication mapping $L\otimes_M L\to L$ induces an $L\times L^{op}$-equivariant surjection $\pi_0(T)\to L$.  We aim to show that it is an isomorphism.

\begin{Lemma}\label{l:hnnlike.bi.iso.comp}
The multiplication mapping $L\otimes_M L\to L$ induces an $L\times L^{op}$-equivariant isomorphism of $\pi_0(T)$ with  $L$.
\end{Lemma}
\begin{proof}
We first prove by induction on the length of $x$ as a product of elements of $M\cup\{t\}$ that there is a path from $[1,x]_M$ to $[x,1]_M$.  If $x=1$ , there is nothing to prove.  Otherwise, assume $x=uy$ with $u\in M\cup \{t\}$ and $y$ of shorter length.  Let $p$ be a path from $[1,y]_M$ to $[y,1]_M$.  Then $up$ is a path from $[u,y]_M$ to $[x,1]_M$.  If $u\in M$, then $[u,y]_M=[1,x]_M$ and we are done. If $u=t$, then $[1,y]_A$ is an edge connecting $[1,x]_M=[1,ty]_M$ to $[t,y]_M=[u,y]_M$ and so we are again done.

Now if $x=uv$ in $L$, then by the above, there is a path $p$ from $[1,u]_M$ to $[u,1]_M$.  Then $pv$ is a path from $[1,x]_M$ to $[u,v]_M$.  If follows that all vertices $[u',v']_M$ with $u'v'=x$ are in a single connected component and hence the multiplication map induces an isomorphism from $\pi_0(T)$ to $L$.
\end{proof}

Next we use derivations to prove that $T$ is a forest.

\begin{Lemma}\label{l:hnnlike.forest.bi}
The graph $T$ is a forest.
\end{Lemma}
\begin{proof}
It suffices to prove that the cellular boundary map $\partial\colon \mathbb ZE\to \mathbb ZV$ is injective.  Define a mapping $\gamma\colon M\cup \{t\}\to \mathbb ZE\bowtie L$ by $\gamma(m) = (0,m)$ for $m\in M$ and $\gamma(t) = ([1,1]_A,t)$. If $a\in A$, then we compute
$\gamma(a)\gamma(t) = ([a,1]_A,at) = ([1,\p(a)]_A,t\p(a))=\gamma(t)\gamma(\p(a))$ and hence $\gamma$ extends to a homomorphism $\gamma\colon L\to \mathbb ZE\bowtie L$ splitting the two-sided semidirect product projection.  Thus $\gamma(x)=(d(x),x)$ for some derivation $d\colon L\to \mathbb ZE$ such that $d(m)=0$ for $m\in M$ and $d(t)=[1,1]_A$.  Define $\beta\colon \mathbb ZV\to \mathbb ZE$ by $\beta([x,y]_M) = d(x)y$.  We must verify that $\beta$ is well defined.  If $m\in M$, then $d(xm)y = xd(m)y+d(x)my = d(x)my$ because $d(m)=0$.  This shows that $\beta$ is well defined.  Next we compute that
\begin{align*}
\beta\partial([x,y]_A) &= \beta([xt,y]_M)-\beta([x,ty]_M) = d(xt)y-d(x)ty\\ &= xd(t)y+d(x)ty-d(x)ty= x[1,1]_Ay=[x,y]_A
\end{align*}
 as $d(t)=[1,1]_A$.  This establishes that $\beta\partial=1_{\mathbb ZE}$ and hence $T$ is a forest.
\end{proof}

We call $T$ the \emph{Bass--Serre forest} for $L$.

The exactness of the sequence
 \[0\longrightarrow \mathbb ZE\longrightarrow \mathbb ZV\longrightarrow H_0(T)\longrightarrow 0,\] coming from $T$ being a forest, together with the isomorphism $\mathbb ZL\cong \mathbb Z\pi_0(L)\cong H_0(T)$ coming from Lemma~\ref{l:hnnlike.bi.iso.comp}, yields the following exact sequence.

 \begin{Cor}\label{c:low.exact.hnnlike}
 Let $L$ be the HNN-like Otto-Pride extension associated to a monomorphism $\p\colon A\to M$ with $A$ a submonoid of $M$.  Then there is an exact sequence
 \[0\longrightarrow \mathbb ZL\otimes_{\mathbb ZA} \mathbb ZL\longrightarrow \mathbb ZL\otimes_{\mathbb ZM} \mathbb ZL\longrightarrow \mathbb ZL\longrightarrow0\]
 where $\mathbb ZL$ is viewed as a right $\mathbb ZA$-module via the inclusion and as a left $\mathbb ZA$-module via $\p$.
 \end{Cor}

Suppose that we have an HNN-like Otto-Pride extension $L$ with base monoid $A$ and monomorphism $\p\colon A\to M$. Put $B=\p(A)$.

\begin{Prop}\label{p:hhnlike.bi.free}
If $M$ is free as a right $A$-set and as a left $B$-set, then $L$ is free as both a right and a left $M$-set.  Moreover, $L$ is free as a right $A$-set and a left $B$-set. Hence $L$ is free as a left $A$-set via the action $a\odot x=\p(a)x$ for $a\in A$ and $x\in L$. 
\end{Prop}
\begin{proof}
This follows from Corollary~\ref{c:normal.form} and its dual.
\end{proof}

The flat version is the following.

\begin{Prop}\label{p:hhnlike.bi.flat}
If $\mathbb ZM$ is a flat right $\mathbb ZA$-module and a flat left $\mathbb ZB$-module, then $\mathbb ZL$ is flat as both a right and a left $\mathbb ZM$-module.  Furthermore, $\mathbb ZL$ is flat as a right $\mathbb ZA$-module and a left $\mathbb ZB$-module. Thus $\mathbb ZL$ is flat as a left $\mathbb ZA$-module via the $\mathbb ZA$-module structure coming from $\p$.
\end{Prop}
\begin{proof}
This follows from Corollary~\ref{c:flat.hnnlike} and its dual.
\end{proof}

We can now investigate the two-sided topological and homological finiteness of HNN-like Otto-Pride extensions.
The following theorem generalises~\cite[Theorem~1]{Pride2004} and~\cite[Theorem~5]{Pride2005}.

\begin{Thm}\label{t:op.hnn.bi}
Let $L$ be an HNN-like Otto-Pride extension of $M$ with respect to an injective homomorphism $\p\colon A\to M$ and put $B=\p(A)$.  Suppose that $M$ is free as a right $A$-set and as a left $B$-set.  Then if $M$ is of type bi-$\F_n$ and $A$ is of type bi-$\F_{n-1}$, then $L$ is of type bi-$\F_n$.
\end{Thm}
\begin{proof}
Let $X$ be a bi-equivariant classifying space for $M$ with $M$-finite $n$-skeleton and $Y$ a bi-equivariant classifying space for $A$ with $A$-finite $(n-1)$-skeleton.  Let $r\colon M\to \pi_0(X)$ and $r'\colon A\to \pi_0(Y)$ be equivariant isomorphisms. By
\cite[Lemma~7.1]{GraySteinberg1}
and the cellular approximation theorem  \cite[Theorem 2.8]{GraySteinberg1}, we can find cellular mappings $f_1,f_2\colon Y\to X$ such that $f_1(aya')=af_1(y)a'$ and $f_2(aya') =\p(a)f_2(y)\p(a')$ for $a,a'\in A$ and $y\in Y$ and, moreover, $r\inv (f_1)_{\ast}r'$ is the inclusion and $r\inv (f_2)_{\ast}r'=\p$ where $(f_i)_{\ast}$ is the induced mapping on the set of path components, for $i=1,2$.

In what follows, we view $L$ as a (free) right $A$-set via the inclusion and a (free) left $A$-set via $\p$.  Put $X'=L\otimes_M X\otimes_M L$ and $Y'=L\otimes_A Y\otimes_A L$.   They are projective $L\times L^{op}$-CW complexes with $L\times L^{op}$-finite $n$-, $(n-1)$-skeletons, respectively, by Proposition~\ref{c:base.change.cw} and~\ref{p:bi.tensor.iso}. Define $F_1,F_2\colon Y'\to X'$ by $F_1(u\otimes y\otimes v) = u\otimes f_1(y)\otimes tv$ and $F_2(u\otimes y\otimes v) = ut\otimes f_2(y)\otimes v$.  Let us verify that this is well defined.  If $a,a'\in A$, then we have that
$ua\otimes f_1(y)\otimes t\p(a')v = ua\otimes f_1(y)\otimes a'tv =u\otimes f_1(aya')\otimes tv$ and so $F_1$ is well defined.  Also, we have that
$uat\otimes f_2(y)\otimes \p(a')v = ut\p(a)\otimes f_2(y)\otimes \p(a')v=ut\otimes \p(a)f_2(y)\p(a')\otimes v=ut\otimes f_2(aya')\otimes v$ and so $F_2$ is well defined.  Clearly, $F_1,F_2$ are continuous $L\times L^{op}$-equivariant cellular mappings.  Let $Z=M(F_1,F_2)$ be the homotopy coequalizer.  It a projective $L\times L^{op}$-CW complex with $L\times L^{op}$-finite $n$-skeleton by construction.  We prove that $Z$ is a bi-equivariant classifying space for $Z$.  To do this it suffices to construct an $L\times L^{op}$-equivariant homotopy equivalence to the Bass--Serre forest $T$.

First note, by~\cite[Proposition~3.4]{GraySteinberg1},
that $\pi_0(X')\cong L\otimes_M M\otimes_M L\cong L\otimes_M L$ (by Proposition~\ref{p:tensor.id}) and $\pi_0(Y')\cong L\otimes_A A\otimes_A L\cong L\otimes_A L$ (by Proposition~\ref{p:basic.tensor.ids}).  The mapping $L\otimes_A L\to L\otimes_M L$ induced by $F_1$ is $u\otimes v\mapsto u\otimes tv$  and the mapping induced by $F_2$ is $u\otimes v\mapsto ut\otimes v$. As $T$ is the homotopy coequalizer of these two mappings of  discrete sets $L\otimes_A L\to L\otimes_M L$, to complete the proof it suffices to show that $X'$ and $Y'$ are homotopy equivalent to their sets of path components (via the natural projections).  But this follows because $X$ and $Y$ are homotopy equivalent to their respective sets of path components and the isomorphisms $X'\cong L/M\times X\times M\backslash L$ and $Y'\cong L/A\times Y\times B\backslash L$ coming from $L$ being free as both a left and right $M$-set and as a right $A$-set and left $B$-set (cf.~Proposition~\ref{p:hhnlike.bi.free}).
\end{proof}

The hypotheses of Theorem~\ref{t:op.hnn.bi} hold if $M$ and $A$ are groups or, more generally, if $M$ is cancellative and $A$ is a group. It also holds if $M=\mathbb N$ and $A$ is a cyclic submonoid.  The proof of Theorem~\ref{t:op.hnn.bi} shows that if $M$ is free as a right $A$-set and a left $B$-set, $M$ has geometric dimension $d$ and $A$ has geometric dimension $d'$, then $L$ has geometric dimension at most $\max\{d,d'+1\}$.

The flat homological analogue of Theorem~\ref{t:op.hnn.bi} has a similar proof.

\begin{Thm}\label{t:op.hnn.bi.flat}
Let $L$ be an HNN-like Otto-Pride extension of $M$ with respect to a monomorphism $\p\colon A\to M$ and put $B=\p(A)$.  Assume that $\mathbb ZM$ is flat as a right $\mathbb ZA$-module and as a left $\mathbb ZB$-module.  If $M$ is of type bi-$\FP_n$ and $A$ is of type bi-$\FP_{n-1}$, then $L$ is of type bi-$\FP_n$.
\end{Thm}
\begin{proof}
We have that $\mathbb ZL$ is flat as both a right and a left $\mathbb ZA$-module and as a right and a left $\mathbb ZM$-module by Proposition~\ref{p:hhnlike.bi.flat} (viewing $L$ as a left $A$-module via $\p$).
Therefore, $\mathbb Z[L\times L^{op}]\cong \mathbb ZL\otimes \mathbb ZL^{op}$ is flat as both a right $\mathbb Z[M\times M^{op}]$-module and as a right-$\mathbb Z[A\times A^{op}]$-module by Proposition~\ref{p:flatness.bi}. Applying Lemma~\ref{l:flat.base} and the hypotheses, we conclude that $\mathbb Z[L\times L^{op}]\otimes_{\mathbb Z[M\times M^{op}]} \mathbb ZM$ is of type $\FP_n$ and $\mathbb Z[L\times L^{op}]\otimes_{\mathbb Z[A\times A^{op}]}\mathbb ZA$ is of type $\FP_{n-1}$.  The result now follows by applying Corollary~\ref{c:fp.resolved} to the exact sequence in Corollary~\ref{c:low.exact.hnnlike}, taking into account Proposition~\ref{p:tensor.id}, Proposition~\ref{p:bi.tensor.iso} and Proposition~\ref{p:basic.tensor.ids},.
\end{proof}

As an example, if $M$ is any group containing a copy of $\mathbb Z$ and $A=\mathbb N$, viewed as a submonoid of $M$, then since $\mathbb ZM$ is free as a module over the group ring of $\mathbb Z$, which in turn is flat over the monoid ring of $\mathbb N$, being a localization, we conclude that $\mathbb ZM$ is flat over the monoid ring of $\mathbb N$.

One can similarly prove that if $L$ is an HNN-like Otto-Pride extension of $M$ with respect to a monomorphism $\p\colon A\to M$ and $\mathbb ZM$ is flat as a right $\mathbb ZA$-module and as a left $\mathbb ZB$-module, where $B=\p(A)$, then if $M$ has Hochschild cohomological dimension at most $d$ and $A$ has Hochschild cohomological dimension at most $d-1$, then $L$ has Hochschild cohomological dimension at most $d$.

We end this section by briefly explaining what happens for a different HNN extensions of monoids construction of the sort considered by Howie~\cite{Howie1963}. Suppose that $M$ is a monoid and $A,B$ are isomorphic submonoids via an isomorphism $\p\colon A\to B$. Let $C$ be an infinite cyclic group generated by $t$.  The \emph{HNN extension} of $M$ with base monoids $A,B$ is the quotient $L$ of the free product $M\ast C$ by the congruence generated by the relations $at=t\p(a)$ for $a\in A$.  In other words, $L=\langle M,t,t\inv\mid tt\inv=1=t\inv t, at=t\p(a), \forall a\in A\rangle$.
The following results may be proved in a similar way to Theorems~\ref{t:ottopride.one.side} and Theorem~\ref{t:op.hnn.bi}, respectively, using suitably modified definition of Bass-–Serre tree, and Bass-Serre forest, for these contexts. 

\begin{Thm}\label{t:hnn.official.one.side}
Let $L$ be an HNN extension of $M$ with base monoids $A,B$.  Suppose that, furthermore, $M$ is free as both a right $A$-set and a right $B$-set.  If $M$ is of type left-$\F_n$ and $A$ is of type left-$\F_{n-1}$, then $L$ is of type left-$\F_n$.
\end{Thm}

\begin{Thm}\label{t:hnn.official.two.side}
Let $L$ be an HNN extension of $M$ with base monoids $A,B$.  Suppose that, furthermore, $M$ is free as both a right and a left $A$-set (via the inclusion) and as a right and a left $B$-set.  If $M$ is of type left-$\F_n$ and $A$ is of type bi-$\F_{n-1}$, then $L$ is of type bi-$\F_n$.
\end{Thm}

Theorem~\ref{t:hnn.official.one.side} recovers the usual topological finiteness result for HNN extensions of groups. It also applies if $M$ is left cancellative and $A$ is a group.  The analogue of Theorem~\ref{t:hnn.official.one.side} for left geometric dimensions states that if $M$ is free as both a right $A$-set and a right $B$-set, $M$ has left geometric dimension at most $d$ and $A$ has geometric dimension at most $d-1$, then $L$ has geometric dimension at most $d$.  Theorem~\ref{t:hnn.official.two.side}  applies if $M$ is  cancellative and $A$ is a group.  Similarly, if $M$ is free as both a right and a left $A$-set and as a right and a left $B$-set, then if $M$ has geometric dimension at most $d$ and $A$ has geometric dimension at most $d-1$, then $L$ has geometric dimension at most $d$.

\end{document}